\documentclass[12pt,reqno]{amsart}
\usepackage[latin1]{inputenc}
\usepackage{amscd}
\usepackage{amsmath,amsthm,amssymb,amsfonts}
\usepackage{mathrsfs}
\usepackage{graphicx}
\usepackage{fancyhdr}
\usepackage{stmaryrd}

\usepackage{color}
\usepackage{amsmath}
\usepackage{amsfonts}
\usepackage{amssymb}
\usepackage{geometry}
\numberwithin{equation}{section}

\newtheorem{theorem}{Theorem}[section]
\newtheorem{proposition}{Proposition}[section]
\newtheorem{lemma}{Lemma}[section]

\newtheorem{corollary}{Corollary}[section]
\theoremstyle{Definition}
\newtheorem{definition}{Definition}[section]

\newtheorem{example}{Example}[section]

\def \D {\mathcal{D}}

\def\R{\mathbb{R}}
\def\C{\mathcal{C}}

\def \H {\mathcal{H}}
\def \p{\partial}

\def \e {\varepsilon}

\def \a {\alpha}
\def \RE {\text{Re}}
\def \IM {\text{Im}}
\def \O {\Omega}

\def \M {\mathcal{M}}
\def \ph {\varphi}
\DeclareMathOperator \inte{int}

\DeclareMathOperator{\dist}{dist}
\DeclareMathOperator{\dive}{div}

\DeclareMathOperator{\supp}{supp}
\DeclareMathOperator{\tr}{tr}
\DeclareMathOperator{\curl}{curl}
\author{Rémy Rodiac}
\title[Description of vorticities of the (G.L) equations]{Description of limiting vorticities for the magnetic 2D Ginzburg-Landau Equations}
\date{}
\address{
Universit\'e catholique de Louvain, Institut de Recherche en Math\'ematique et Physique, Chemin du Cyclotron 2 bte L7.01.01, 1348 Louvain-la-Neuve, Belgium}
\email{remy.rodiac@uclouvain.be}
\begin{document}

\begin{abstract}
Let $\O$ be a bounded open set in $\R^2$. The aim of this article is to describe the functions $h$ 
in $H^1(\O)$ and the Radon measures $\mu$ which satisfy $-\Delta h+h=\mu$ and $
\dive(T_h)=0$ in $\O$, where $T_h$ is a $2\times 2$ matrix given by $(T_h)_{ij}=2\p_ih\p_jh-
(|\nabla h|^2+h^2)\delta_{ij}$ for $i,j=1,2$. These equations arise as equilibrium 
conditions satisfied by limiting vorticities and limiting induced magnetic fields of 
solutions of the magnetic Ginzburg-Landau equations. This was shown by Sandier-Serfaty in 
\cite{SandierSerfaty2003,SandierSerfaty2007}. Let us recall that they obtained that $|\nabla h|$ is continuous in $\O$. We prove that if $x_0$ in $
\O$ is in the support of $\mu$ and is such that $|\nabla h(x_0)|\neq 0$ then $\mu$ is absolutely continuous with respect to 
the 1D-Hausdorff measure restricted to a $\C^1$-curve near $x_0$ whereas $\mu_{\lfloor \{|
\nabla h|=0\}}=h_{|\{ |\nabla h|=0\}} \mathcal{L}^2$. We also prove that if $\O$ is smooth bounded and star-shaped and if $h=0$ on $\p \O$ then $h 
\equiv 0$ in $\O$. This rules out the possibility of having critical points of the Ginzburg-Landau energy with a number 
of vortices much larger than the applied magnetic field $h_{ex}$ in that case.

\end{abstract}

\maketitle

\section{Introduction}
\subsection{Statement of the problem and main results}
Let $\O \subset \R^2$ be a bounded open set. In this article we are interested in describing the functions $h$ in $H^1(\O)$ and the Radon measures $\mu$ in $\mathcal{M}(\O)$ which satisfy
\begin{equation}\label{1}
-\Delta h+h=\mu \ \ \text{ in } \O,
\end{equation}
\begin{equation}\label{2}
\sum_{i=1}^{2} \p_i\left[ 2\p_ih\p_jh-\left(|\nabla h|^2+h^2\right)\delta_{ij}\right]=0 \ \ \text{ in } \O \text{ for } \ \ j=1,2.
\end{equation}
We are particularly interested in the description of the support of $\mu$, i.e., the complement of the largest open set on which $\mu$ vanishes. In the majority of the paper we only look at local regularity questions but when needed we prescribe the boundary conditions
\begin{equation}\label{BC}
h\equiv 0 \ \ \text{ on } \p \O \ \ \ \text{ or } \ \ \ \  h\equiv 1 \ \ \text{ on } \p \O.
\end{equation}
These are the natural boundary conditions for this problem when related to the Ginzburg-Landau (G.L) theory. Equation \eqref{2} can be rewritten as 
\begin{equation}
\dive (T_h)=0 \text{ in } \O ,
\end{equation}
where $T_h$ is  the so-called \textit{stress-energy tensor} associated to the functional
\begin{equation}\label{F}
F(h)=\int_\O \left( |\nabla h|^2+h^2 \right)
\end{equation}
and is defined by 
\begin{eqnarray}\label{eq:defTh}
T_h &= & \begin{pmatrix}
(\p_xh)^2-(\p_yh)^2 -h^2 & 2(\p_xh) (\p_yh) \\
2(\p_xh) (\p_yh) & (\p_yh)^2-(\p_xh)^2-h^2
\end{pmatrix}.
\end{eqnarray}
This can also be written as $T_h=2\nabla h \otimes\nabla h -(|\nabla h|^2+h^2)I_2.$ We remark that, since $h$ is in $H^1(\O)$, then $T_h$ is in $L^1(\O)$. Thus Equation \eqref{2} is well-defined in the sense of distributions as $\int_\O (T_h)_i^T \cdot \nabla \varphi=0$, for all $\varphi$ in $\C^\infty_c(\O)$, for $i=1,2$, and where $(T_h)_i$ is the i-th line of the matrix $T_h$. Equation \eqref{2} means that $h$ is a \textit{stationary point} of $F$. More precisely it is a critical point of $F$ with respect to inner variations or variations of the domain, i.e.,
\begin{equation}
\frac{d}{dt}|_{t=0}F(h(x+tX(x))=0
\end{equation}
for all $X$ in $\C^\infty_c(\O,\R^2)$. We note that for $t$ small enough $x \mapsto x+tX(x)$ is a diffeomorphism of $\O$. These inner variations lead to different equations from the classical Euler-Lagrange equations given by  outer variations: $h_t(x)=h(x)+t\varphi(x)$ for $\varphi$ in $\C^\infty_c(\O)$. However it is known that if $h$ is a a critical point of some functional $G$ and $h$ is smooth enough-- $\C^2$ is sufficient-- then it is also critical for the inner variations. Indeed, we can then write an inner variation as $h(x+tX(x))= h(x)+t\nabla h(x)\cdot X(x)+O(t^2)$, with $O(t^2)\leq Mt^2$ uniformly in $\O$, see e.g.\ p.30-31 in \cite{Rivierecours}. Thus our first example of $h$ and $\mu$ satisfying \eqref{1}-\eqref{2} is
\begin{example}
Let $h$ be in $H^1(\O)$ such that $-\Delta h+h=0$ then $h$ satisfies \eqref{1} and \eqref{2} with $\mu=0$. Indeed, by ellipic regularity theory, $h$ is smooth and hence satisfies \eqref{2}.
\end{example}
Conversely we can show that if $h$ is $\C^2$ and satsifies \eqref{2} then $-\Delta h+h=0$ or $h$ is constant in $\O$. The argument is the following: by applying the chain rule in \eqref{2} we find that $(-\Delta h +h)\nabla h=0$ in $\O$. Thus $-\Delta h+h=0$ in $\{|\nabla h|\neq 0 \}$. But we also have that $\Delta h=0$ almost everywhere on $\{ |\nabla h|=0\}$, cf.\ e.g.\ Lemma 7.7 in \cite{GilbargTrudinger}. Hence $-\Delta h+h=h\textbf{1}_{\{ |\nabla h|=0 \}}$. By using the Morse-Sard Theorem, $h$ is constant on the connected components of $\{|\nabla h|=0 \}$. But then $h\textbf{1}_{\{|\nabla h|=0 \}}$ is continuous in $\O$ if and only if $h=0$ on $\{|\nabla h|=0 \}$ or $\{|\nabla h|=0 \}$ is either empty or equal to $\O$. This implies that $-\Delta h +h=0$ in $\O$ or $h=cst$ in $\O$. For a more general statement about the equivalence between the inner-variational equations and Euler-Lagrange equations for smooth solutions we refer to \cite{Faliagas2016}. For an example of a variational problem where there exists a (non-smooth) critical point for the outer variations which is not critical for the inner variations we refer to \cite{Riviere1995}. The regularity theory for solutions of inner variational equations is rather limited until now and ``it is still in its infancy" to use the same words as in \cite{IwaniecKovalevOnninen2013} to which we refer for more on this subject.

Before stating our main results, let us give non-trivial examples of $h$ and $\mu$ satisfying \eqref{1}-\eqref{2}.
\begin{example}\label{ex:example1}
Let $\O:= (-1,1)\times (-1,1)$ and $h(x,y)=f(x)$ in $\O$ where
\begin{equation}
f(x):=\begin{cases} e^x & \text{ if }   -1<x\leq 0 \\
e^{-x} & \text{ if  }  \phantom{-a}0\leq x<1.
\end{cases}
\end{equation}
We have that $h$ is in $H^1(\O)$, and since $|f'|=|f|$ we can check that $h$ satsifies \eqref{2}. Furthermore we have that $-\Delta h+h=\mu$ where $\mu= -2\mathcal{H}^1_{\lfloor \{ x=0\}}=-2\frac{\p h}{\p x^+}(0,y)\H^1_{\lfloor \{x=0 \}}$. 
\end{example}
In this example we see that $h$ is only Lipschitz and $\mu$ is concentrated on a curve. We used the notation $\p h /\p x^+(x_0,y_0)=\lim_{x\rightarrow x_0^+}\frac{h(x,y_0)-h(x_0,y_0)}{x-x_0}$, and $\p h/\p x^-$ is defined in an analogous way. We also remark that, $\p h/\p x^+(0,y)=-\p h/\p x^-(0,y)$ for all $y$ in $(-1,1)$ and that $|\nabla h|=|f'|$ is continuous in $\O$. The second example is more regular.
\begin{example}\label{ex:example2}
Let $\O$ be any smooth open bounded domain in $\R^2$, let $\lambda>0$ and let $h$ in $H^1(\O)$ be the solution of the following \textit{obstacle problem}:
\begin{equation}
\min \{ \frac{1}{2} \int_{\O} \left( |\nabla h|^2 +h^2 \right ); h=1 \text{ on } \p \O \text{ and } h\geq 1-\frac{1}{2\lambda} \} .
\end{equation}
Then, since inner variations preserve the space of minimization, we can verify that $h$ satsifies \eqref{2}. Besides it is known, cf.\ \cite{Frehse1972,Caffarelli1998}, that $h$ is $\C^{1,1}(\O)$ and $-\Delta h+h=\left( 1-\frac{1}{2\lambda} \right)\textbf{1}_{\omega_\lambda}$ where $\omega_\lambda$ is the coincidence set $\omega_\lambda:=\{x\in \O; h(x)=1-\frac{1}{2\lambda}\}$. For $\lambda$ large enough this coincidence set is not empty, see e.g.\ Proposition 7.2 in \cite{SandierSerfaty2007}, and thus the measure $-\Delta h+h$ is non trivial.
\end{example}
In the above example it can be checked that we also have $\mu=h\textbf{1}_{\{|\nabla h|=0 \}}$. Our main result says that every solution of \eqref{1}-\eqref{2} is a combination of the two previous examples. First we recall that if $h$ is in $H^1(\O)$ and satisfies \eqref{2} then $|\nabla h|$ is continous, cf.\ Theorem 13.2 in \cite{SandierSerfaty2007} or Lemma 4.1 \cite{SandierSerfaty2003}. Thus it makes sense to talk about the set $\{ |\nabla h|>0\}$ and its complementary.

\begin{theorem}\label{th:main1}
Let $h$ be in $H^1(\O)$ and $\mu$ be in $\M(\O)$ such that \eqref{1}, \eqref{2} hold. Then $\supp \mu \cap \{|\nabla h|>0 \}$ is locally $\mathcal{H}^1$-rectifiable, $\nabla h$ is in $BV_{\text{loc}}(\O)$ and 
\begin{itemize}
\item[1)] there exists  $\sigma: \supp \mu \cap \{|\nabla h|>0 \}\rightarrow \{\pm 1 \}$ such that \begin{equation}\label{eq:decomposition*}
\mu=h\textbf{1}_{\{|\nabla h|=0\}}+2\sigma (x)|\nabla h |\mathcal{H}^1_{\supp \mu \cap \{|\nabla h|>0 \}},
\end{equation}
with $\epsilon$ which is constant on every connected component of $\supp \mu \cap \{|\nabla h|>0 \}$.
\item[2)] The function $h$ is constant on the connected components of $\supp \mu$.
\end{itemize}
\end{theorem} 
Besides a description of the measure $\mu$ we also provide a non-existence result when the boundary data is $h= 0$ on $\p \O$.

\begin{theorem}\label{th:main2}
Let $\O \subset \R^2$ be a smooth, bounded and star-shaped domain. Let $h$ be in $H^1(\O)$ which satisfies \eqref{2} and such that $h= 0$ on $\p\O$ then $h=0$ in $\O$.
\end{theorem}

Theorems \ref{th:main1} and \ref{th:main2} allow us to say something on the repartition and on the number of vortices of solutions of the (G.L) equations and to partially answer two open problems in \cite{SandierSerfaty2007}. We now explain the connection between the (G.L) theory and the critical conditions \eqref{1} and \eqref{2}.

\subsection{Motivation: the magnetic (G.L) equations}

 In this section we assume that $\O \subset \R^2$ is a smooth bounded simply connected domain. Conditions \eqref{1} and \eqref{2} are satisfied by limiting vorticities and limiting induced magnetic fields of the (G.L) equations. The (G.L) energy was introduced in the 50's by Ginzburg and Landau to describe the behavior of a superconductor material. This energy is 
\begin{equation}\label{Ge}
G_\e(u,A)=\frac{1}{2}\int_\O \left(|\nabla u-iAu|^2+\frac{1}{2\e^2}(1-|u|^2)^2+|h-h_{ex}|^2 \right).
\end{equation}
We also set 
\begin{equation}
G_\e^0(u,A)=\frac{1}{2}\int_\O \left(|\nabla u-iAu|^2+\frac{1}{2\e^2}(1-|u|^2)^2+h^2 \right).
\end{equation}
Here $u:\O \rightarrow \mathbb{C}$ is called the order parameter, $h_{ex}$ is the intensity of the exterior magnetic field, it is a constant in $\O$ depending on $\e$, $A:\O \rightarrow \R^2$ is the vector potential of the induced magnetic field $h$ obtained through the formula $h=\curl A=\p_1A_2-\p_2A_1$ and $\e>0$ is a parameter. We are interested here in type-II superconductors, which correspond to $\e$ small. In this type of material, for large values of the applied magnetic field $h_{ex}$ the superconductivity is destroyed in small regions of the sample surronding by superconducting currents: these are the \textit{vortices}. They can be thought as zeros of the order parameter around which $u$ has a non-zero degree, for the definition of the degree see e.g.\ \cite{SandierSerfaty2007}. The main question is to understand the number and the repartion of these vortices in the limit $\e\rightarrow 0$, called the \textit{London limit}. We refer to \cite{BethuelBrezisHelein1994,SandierSerfaty2007} and references therein for a more complete description of this problem. We are interested in critical points $(u_\e,A_\e)_{\e>0}$ of \eqref{Ge} in the space
\begin{equation}
X=\{ (u,A)\in H^1(\O,\mathbb{C})\times H^1(\O,\R^2); (\curl A-h_{ex})\in L^2(\O)\}. 
\end{equation} 
They satisfy the following Euler-Lagrange equations
\begin{equation}\label{eq:GL}
\left\{
\begin{array}{rcll}
-(\nabla-iA)^2u &=& \frac{1}{\e^2}u(1-|u|^2)  & \text{ in } \O \\
-\nabla^\perp h &=& (iu,\nabla u-iAu) & \text{ in } \O \\
h &=& h_{ex} & \text{ in } \R^2 \setminus \O \\
\nu\cdot (\nabla-iA)u &=&0 & \text{ on } \p \O.
\end{array}
\right.
\end{equation}
Here $\nabla^\perp h=(-\p_2h,\p_1h)^T$, $\nu$ is the outer normal to $\p \O$ and $(\cdot,\cdot)$ is the scalar product in $\mathbb{C}$ identified with $\R^2$, i.e., $(a,b)=\frac{1}{2}(\bar{a}b+a\bar{b})$. The problem is to understand the limiting distributions of the zeros of $u_\e$ as $\e$ tends to zero. In order to do that the important quantity to consider is the \textit{vorticity} $\mu_\e$ defined by
\begin{equation}\label{def:vorticity}
\mu_\e:=\curl (iu_\e,\nabla u_\e-iA_\e u_\e)+\curl A_\e.
\end{equation}
Taking the curl of the second equation in \eqref{eq:GL} we see that the induced magnetic field and the vorticity are linked by the equation
\begin{equation}\label{eq:vorticities}
-\Delta h_\e+h_\e=\mu_\e.
\end{equation}
In \cite{SandierSerfaty2007} Sandier-Serfaty proved that for $\e>0$ small, if $G_\e^0(u_\e,A_\e) \leq C\e^{-\alpha}$ with $\alpha<1/3$, we have that 
\begin{equation}
\mu_\e \simeq 2\pi \sum_{i=1}^{N_\e}d_i^\e \delta_{x_i^\e}
\end{equation}
where $x_i^\e$ are essentially the center of the vortices, $d_i^\e$ their degrees, $N_i^\e$ the number of vortices and the meaning of $\simeq$ will be precised in Theorem \ref{th:SandierSerfaty1}. Equation \eqref{eq:vorticities} when $\mu_\e$ is a sum of Dirac masses is called the \textit{London} equation in physics. To understand the limiting distributions of the vortices we can then look at the weak limit of $\mu_\e/n_\e$ where $n_\e:=\sum_{i=1}^{N_\e}|d_i^\e|$. It is expected that this limiting vorticity satisfies some equilibrium conditions due to the fact that two vortices of opposite sign attract each other whereas two vortices of the same sign repel each other. In absence of magnetic field and with a fixed boundary data with non-zero degree, Bethuel-Brezis-H\'elein showed that vortices of solutions of the (G.L) equations converge to critical points of a \textit{renormalized energy}. Conditions \eqref{1}-\eqref{2} can be viewed as an analogue of  the condition of being a critical point of a renormalized energy in the case with magnetic field. To derive these equilibrium conditions Sandier-Serfaty passed to the limit in the conservative form of the (G.L) equations. The heuristic of the argument is the following: since a critical point of the (G.L) equations is smooth it is also a stationary point. Hence it satisfies
\begin{equation}
\dive (S_{u_\e,A_\e})=0 \text{ in } \O,
\end{equation}
where $S_{u_\e,A_\e}$ is the stress-energy tensor associated to the (G.L) energy and is expressed by ${(S_{u_\e,A_\e})}_{ij}= 2(\p_i^Au,\p_j^Au)-\left(|(\nabla-iA)u|^2-h^2+\frac{1}{2\e^2}(1-|u|^2)^2\right)\delta_{ij}$ for $i,j=1,2$ and where $\p_j^Au=\p_ju-iA_ju$. Formally by using the second equation in \eqref{eq:GL} and the fact that $|u_\e|\simeq 1 $ when $\e$ is small we can see that 
\begin{equation}
S_{u_\e,A_\e} \simeq T_{h_\e} \text{ in } \O. 
\end{equation}
Then passing formally to the limit we obtain the equilibrium condition $\dive(T_h)=0$ which is \eqref{2}. We note that to pass to the limit, a priori we need the stong convergence of $h_\e/n_\e$ in $H^1(\O)$. The weak convergence of $h_\e$ in $H^1(\O)$, or some strong convergence of $h_\e/n_\e$ in $W^{1,p}(\O)$ for some $1<p<2$ is not sufficient a priori. But using some compensation properties analogous to a Theorem of Delort \cite{Delort1991}, see  also DiPerna-Majda \cite{DiPernaMajda1988}, Sandier-Serfaty made rigorous the formal argument above. We also remark that conditions \eqref{1} and \eqref{2} are reminiscent of a formal mean-field model derived by Chapman-Rubinstein-Schatzman in \cite{ChapmanRubinsteinSchatzman1996}. We can now cite the results on the limiting vorticities problem that we divide into two parts.

\begin{theorem}\label{th:SandierSerfaty1}(Th.1.7  and Th.13.1 in \cite{SandierSerfaty2007})
Let $(u_\e,A_\e)_{\e>0}$ be solutions of the (G.L) equations such that $G_\e^0(u_\e,A_\e)\leq C\e^{-\a}$ with $\a<1/3$. Then for any $\e>0$ there exists a measure $\nu_\e$ of the form $2\pi\sum_i d_i^\e\delta_{x_i^\e}$ where the sum is finite, $x_i^\e\in \O$ and $d_i^\e\in \mathbb{Z}$ for every i, such that letting $n_\e=\sum_i|d_i^\e|$,
\begin{equation}
n_\e \leq C\frac{ G_\e^0(u_\e,A_\e)}{|\log \e|},
\end{equation}
\begin{equation}
\|\mu_\e-\nu_\e \|_{W^{-1,p}(\O)} \|\mu_\e-\nu_\e \|_{\M(\O)}\rightarrow 0,
\end{equation}
for some $p\in (1,2)$.
\end{theorem}

\begin{theorem}(Th.1.7  and Th.13.1 in \cite{SandierSerfaty2007}, Th.1 in \cite{SandierSerfaty2003})\label{th:SandierSerfaty2}
With the same notations as the previous theorem, possibly after extraction, one of the following holds.
\begin{itemize}
\item[0)] $n_\e=0$ for every $\e$ small enough and then $\mu_\e$ tends to $0$ in $W^{-1,p}(\O)$.
\item[1)] $n_\e=o(h_{ex})$, and then, for some $p\in (1,2)$, $\mu_\e/n_\e$ converges in $W^{-1,p}(\O)$ to a measure $\mu$ such that 
\begin{equation}
\mu\nabla h_0=0,
\end{equation}
where $h_0$ is the solution of $-\Delta h_0+h_0=0$ in $\O$, $h_0=1$ on $\p \O$. Hence $\mu$ is a linear combination of Dirac masses supported in the finite set of critical points of $h_0$.
\item[2)]$ h_{ex}\simeq \lambda n_\e$, with $\lambda>0$, then for some $p\in (1,2)$, $\mu_\e/h_{ex}$ converges in $W^{-1,p}(\O)$ to a measure $\mu$ and $h_\e/h_{ex}$ converges strongly in $W^{1,p}(\O)$ to the solution of 
\begin{equation}
\left\{
\begin{array}{rcll}
-\Delta h +h &=&\mu & \text{ in } \O \\
h&=&1 & \text{ on } \p \O.
\end{array}
\right.
\end{equation}
Moreover the tensor $T_h$ is divergence-free in finite part.
\item[3)] $h_{ex}=o(n_\e)$ and then for some $p\in (1,2)$, $\mu_\e/n_\e$ converges in $W^{-1,p}(\O)$ to a measure $\mu$ and $h_\e/n_\e$ converges strongly in $W^{1,p}(\O)$ to the solution of 
\begin{equation}
\left\{
\begin{array}{rcll}
-\Delta h +h &=&\mu & \text{ in } \O \\
h&=&0 & \text{ on } \p \O.
\end{array}
\right.
\end{equation}
Moreover the tensor $T_h$ is divergence-free in finite part.
\end{itemize}
In cases 2) and 3) if the limit $h$ of $h_\e/n_\e$ is in $H^1(\O)$ then $T_h$ divergence-free in finite part is equivalent to the condition \eqref{2}.
\end{theorem}

The previous theorems describe the limiting vorticities in the so-called \textit{mean-field regime}. For the definition of \textit{divergence-free in finite part} we refer to Chapter 13 of \cite{SandierSerfaty2007}. For the purpose of this article we only need to know that if $X$ is in $L^1_{\text{loc}}(\O,\R^2)$ then $X$ is divergence-free in finite part if and only if $\dive(X)=0$ in the sense of distributions. The previous result leads to the open problem --Open problem 14 in \cite{SandierSerfaty2007}-- of describing the $h$ and $\mu$ satisfying the conditions of items 2) and 3). Theorem \ref{th:main1} brings a partial answer to that question. In the case where $\mu$ is in $H^{-1}(\O)$ then $\mu$ can only be supported on some curves or sets of full 2D-Lebesgue measure. Vorticities of \textit{minimizers} of the (G.L) energy are known to converge to a measure $\mu$ absolutely continuous with respect to the Lebesgue measure if $h_{ex}\simeq \lambda |\log \e|$. More precisely we have $\mu=h\textbf{1}_{\{|\nabla h|=0\}}$ where $h$ is the solution of the obstacle problem, cf.\ Example 3 and Chapter 7 of \cite{SandierSerfaty2007}. Contreras-Serfaty in \cite{ContrerasSerfaty2012} constructed local-minimizers of the Ginzburg-Landau energy with a number of vortices as large as $|\log \e|$ and external magnetic fields as large as $\e^{-1/7}$.  The vorticities of these local minimizers also converge to a measure related to the obstacle problem. In \cite{Aydi2008}, Aydi constructed \textit{non-minimizing} solutions of the (G.L) equations whose vorticities concentrate on lines or on circles as $\e$ tends to zero.  These examples show that both types of measures appearing in the decomposition \eqref{eq:decomposition*} are indeed attained by solutions of the (G.L) equations. 

Since we work under the hypothesis that $h$ is in $H^1(\O)$, we recall that this case happens if we assume that $G_\e(u_\e,A_\e)\leq Ch_{ex}^2$ and $n_\e \simeq \lambda h_{ex}$ for some $\lambda >0$, cf.\ \cite{SandierSerfaty2003}.


Case 3) of Theorem \ref{th:SandierSerfaty2} can only happen if there is a number of vortices much larger than the applied magnetic field. It is an open question--see Open problem 18 in \cite{SandierSerfaty2007}-- to know if there can exist (G.L) solutions with a number of vortices much larger than the exterior magnetic field. Our Theorem \ref{th:main2} says that, in star-shaped domains, this is very likely to be impossible since then the limiting vorticity has to be zero. This indicates that vortices tend to escape the domain in that case.

\subsection{Previous results on the problem}

We now state the previous results obtained by different authors on the regularity of $(h,\mu)$ satisfying \eqref{1} and \eqref{2} and on the non-existence of non-trivial $(h,\mu)$ satisfying $h=0$ on $\p \O$. In \cite{SandierSerfaty2007} the authors proved that if we assume $\mu\in L^p$ for some $p>1$ then we necessarily have $\mu=h\textbf{1}_{\{|\nabla h|=0\}}$. They also obtained the same result under different hypotheses. Indeed in \cite{SandierSerfaty2003} they assumed that $\nabla h\in \C^0(\O)$ and $|\nabla h|\in BV(\O)$ to obtain $\mu=h\textbf{1}_{\{|\nabla h|=0\}}$. Under both set of hypothesis $h$ is also solution of a free-boundary problem studied by Caffarelli-Salazar \cite{CaffarelliSalazar2002} and Caffarelli-Salazar-Shagohlian \cite{CaffarelliSalazarShahgholian}. These authors studied the equation
\begin{equation}\label{eq:Caffarelli}
\Delta h=h \text{ on } \{|\nabla h| \neq 0 \}
\end{equation}
where this equation has to be understood in some viscosity sense assuming only that $h$ is continuous a priori, cf.\ Definition \ref{def:viscositysol1}. They proved that a solution is in fact in $\C^{1,1}(\O)$ and they studied the regularity of the free boundary $\p \{ |\nabla h|>0\}$. In \cite{SandierSerfaty2007}, Sandier-Serfaty also obtained that if $\mu=h\textbf{1}_{\{|\nabla h|=0 \}}$ and $h=0$ on $\p \O$ then $h$ has to be identically zero. 

In \cite{Le2009}, Le studied the conditions \eqref{1} and \eqref{2} under the assumption that $\mu$ is absolutely continuous with respect to the Hausdorff measure restricted to a simple smooth curve $\Gamma$ with a nowhere-zero density which is in $W^{2,p}(\Gamma)$ for some $p>1$. In this case he obtained that $h$ is constant on $\Gamma$. Besides he proved that if the domain is ``thin" enough then there is no nontrivial solution of the problem with $h=0$ on $\p \O$. He also obtained that $\mu$ is a Radon measure with a fixed sign and that if the domain is ``thin" enough then it has to be positive if $h=1$ on $\p \O$. For the appropriate notion of thinness in that case we refer to \cite{Le2009}.

At last a similar, but simpler problem than conditions \eqref{1} and \eqref{2} was studied by the author in \cite{Rodiac2016}. We considered the conditions $h\in H^1(\O)$, $\Delta h=\mu \in \M(\O)$ and $\dive(\tilde{T_h})=0$ where $(\tilde{T_h})_{ij}=2\p_ih\p_jh-|\nabla h|^2 \delta_{ij}$ for $i,j=1,2$. In this problem the condition $\dive(\tilde{T_h})=0$ in $\O$ is equivalent to $(\p_zh)^2$ holomorphic in $\O$, where $2\p_z=\p_1-i\p_2$. The connection with holomorphic functions makes it simpler than the condition \eqref{2}. We obtained that, locally, $\mu$ is supported on some curve or some curves which intersect and these curves form the zero set of a harmonic function or a multi-valued harmonic function. This problem is related to limiting vorticities of solutions of the (G.L) equations without magnetic field but also to the stationary incompressible Euler equations in fluid mechanics and to systems of stationary point-vortices or stationary Helmholtz-Kirchhoff's systems.

\subsection{Plan of the paper and methods of the proof}

In Section \ref{I} we recall that if $h$ satisfies \eqref{2} then $|\nabla h|$ and $T_h$ are continuous. We also recall that if $h$ satisfies \eqref{2} then a Pohozaev formula holds for $h$ and this formula implies the non-existence result Theorem \ref{th:main2}. But we prove that if $\O$ is an annulus then there exists a non trivial solution of \eqref{1}-\eqref{2} such that $h=0$ on $\p \O$.
In Section \ref{II} we study the local behavior of $h$ and $\mu$ satisfying \eqref{1} and \eqref{2} near a point $z_0$ such that $|\nabla h|(z_0)\neq 0$. We obtain that near such a point the measure $\mu$ is absolutely continuous with respect to the 1D-Hausdorff measure restricted to a $\C^1$-curve. In order to do that we use the complex form of Equation \eqref{2} which can be written as 
\begin{equation}\label{eq:2complex}
\p_{\bar{z}}\left[(\p_zh)^2 \right]=\frac{1}{4}\p_z(h^2),
\end{equation}
where we let $\p_{\bar{z}}:=\frac{1}{2}(\p_1+i\p_2)$ and $\p_z:=\frac{1}{2}(\p_1-i\p_2)$. Thanks to this equation we deduce that in a ball of radius $R>0$ around $z_0$ we have $\nabla h= \theta F$ where $\theta \in BV(B_R(z_0),\{\pm 1\})$ and $F\in W^{1,q}(B_R(z_0),\R^2)$ for every $1\leq q<+\infty$. We then employ the method of \cite{Rodiac2016} to obtain that $\supp \mu \cap B_R(z_0)$ is equal to the boundary, in a measure theoretic sense, of $\{\theta=1\}$. A finer analysis, that uses special properties of sets of finite perimeter in $\R^2$ cf.\ \cite{AmbrosioCasellesMasnouMorel2001}, yields that this is a $\C^1$ curve and $h$ is constant on that curve. In Section \ref{III} we prove Theorem \ref{th:main1}. This rests upon the analysis of the previous section. In particular, thanks to sufficiently good estimates obtained in Section \ref{I}, we are able to prove that $\nabla h$ is in $BV_{\text{loc}}(\O)$. We can then adapt an argument of Sandier-Serfaty \cite{SandierSerfaty2003} to prove that $\Delta h_{\lfloor \{|\nabla h|=0 \}}=0$ and obtain the first item of Theorem \ref{th:main1}.
 To prove that $h$ is constant on the connected components of the support of $\mu$, we can use three different arguments. All of them use the local study near the regular point of  $h$ of the previous section and a Morse-Sard type Theorem for Sobolev functions, cf.\ \cite{dePascale2001,Figalli2008}. The first one rests entirely on the description of the measure $\mu$ obtained in the first point of Theorem \ref{th:main1}. The second one does not use this full description but rather fine properties of solutions of $\Delta h\in \M(\O)$ and results of Alberti-Bianchini-Crippa  \cite{AlbertiBianchiniCrippa2014b,AlbertiBianchiniCrippa2014a}. The third proof employs a connection with the problem studied by Caffarelli-Salazar in \cite{CaffarelliSalazar2002}
 
%
 

\subsection{Notations}

Throughout the paper we use the following notations: 
\begin{itemize}
\item[$\bullet$] $B_R(z)=\{ |z-z_0|<R \}$ is the ball of center $z_0$ and radius $R>0$, we also denote it by $B(z,R)$ in some part of the paper;
\item[$\bullet$] $\textbf{1}_{A}$ denotes the characteristic function of a set $A$;
\item[$\bullet$] $\mathcal{M}(\O)$ denotes the Banach space of Radon measures in $\O$, i.e., the Borel measures which are locally finite on $\O$. These measures are not necessarily postive. With some abuse of notations we also use $\mathcal{M}(\O)$ for vector-valued Radon measures;
\item[$\bullet$] if $\mu$ is a Radon measure, $|\mu|$ denotes its total variation defined by $|\mu|(E)= \sup \{\sum_{i=1}^{+\infty}|\mu(E_j)|; E_j \text{ are Borel sets, pairwise disjoints and } E=\bigcup_{i=1}^{+\infty}E_j\}$ for every Borel set $E$;
\item[$\bullet$] for $\mu$ a Radon measure its norm is denoted by $\|\mu\|_{\mathcal{M}(\O)}=|\mu|(\O)= \sup \{\int_\O \varphi d\mu ; \varphi\in \C^0_c(\O), \|\varphi\|_{L^\infty(\O)}\leq 1 \}$; 
\item[$\bullet$] $\mathcal{L}^2$ is the Lebesgue measure on $\R^2$, if $U$ is a Borel set then $|U|$ denotes its Lebesgue measure;
\item[$\bullet$] $\mathcal{H}^k$ is the $k$-dimensional Hausdorff measure;
\item[$\bullet$] If $E$ is a set of finite perimeter in $\O$ we denote by $\p_\star E$ its measure theoretic boundary, also called essential boundary, cf.\ \cite{EvansGariepy2015,AmbrosioFuscoPallara2000};
\item[$\bullet$] if $\mu,\nu$ are two Radon measures with $\nu$ postive in $\O$, $\mu << \nu$ means that $\mu$ is absolutely continuous with respect to $\nu$, i.e., $\mu(A)=0$ if $\nu(A)=0$;
\item[$\bullet$] If $\mu,\nu$ are two positive Radon measures we write $\mu \perp \nu$ to say that they are mutually singular, i.e., there exists a Borel set $E$ such that $\mu(E)=0$ and $\nu(\O \setminus E)=0$.
\item[$\bullet$] If $f$ is in $BV(\O)$, we denote by $Df$ its distributional derivative, which is a Radon measure and by $\nabla f$ the absolute continuous part of $Df$ with respect to the Lebesgue measure.
\item[$\bullet$] We recall that a set $M$ is \textit{countably $k$-rectifiable} if it is $\mathcal{H}^k$-mesurable and $M= \bigcup_{i=0}^{+\infty} M_i$, where $\mathcal{H}^k(M_0)=0$ and $M_i=f_i(A_j)$, for some set $A_j\subset \R^n$ and some Lipschitz functions $f_j:A_j\rightarrow \R$. A set $M$ is $\mathcal{H}^k$-rectifiable if it is countably $\mathcal{H}^k$-rectifiable and $\mathcal{H}^k(M)<+\infty$. A set $M$ is locally $\mathcal{H}^k$-rectifiable if for every $x$ in $M$ there exists $r>0$ such that $B_r(x)\cap M$ is $\mathcal{H}^k$-rectifiable.
\end{itemize}

\section{First regularity results and the non-existence result}\label{I}

\subsection{Preliminaries regularity results} 

It is known that if $\mu$ is only a Radon measure then Equation \eqref{1} is critical for the Calder\'on-Zygmund theory. We cannot expect more regularity than $h$ in $W^{1,p}_{\text{loc}}(\O)$ for every $1\leq p<2$, cf.\ \cite{Ponce2016}. But we assumed that $h$ is in $H^1(\O)$, to give a meaning to \eqref{2}. Then Equation \eqref{2} gives us more regularity.
\begin{proposition}\label{reg1}
Let $h$ be in $H^1(\O)$ which satisfies \eqref{2} then
\begin{equation}
|\nabla h|^2\in W^{1,p}_{\text{loc}}(\O) \text{ for any } 1\leq p <+\infty,
\end{equation}
\begin{equation}
T_h  \in W^{1,p}_{\text{loc}}(\O,\R^4)  \text{ for any } 1\leq p <+\infty,
\end{equation}
where $T_h$ is defined by \eqref{eq:defTh}. In particular we have
\begin{equation}
|\nabla h| \text{ is in } \C^{0,\alpha}(\O) \text{ for every } 0<\alpha<\frac{1}{2},
\end{equation}
\begin{equation}
h \text{ is in } W^{1,\infty}_{\text{loc}}(\O), \text{i.e., h is locally lipschitz}.
\end{equation}
\end{proposition}
This proposition is proved in \cite{SandierSerfaty2007} (proof of Theorem 13.1 p.278). We sketch the argument here for the comfort of the reader. We will use the complex form of the equation \eqref{2} throughout the paper. We let $\p_{\bar{z}}:=\frac{1}{2}(\p_1+i\p_2)$ and $\p_z:=\frac{1}{2}(\p_1-i\p_2)$. Then we can write \eqref{2} as \eqref{eq:2complex}.
\begin{proof}
Since $h$ is in $H^1(\O)$ we can see that $h^2$ is in $W^{1,p}(\O)$ for every $1\leq p< 2$ and $
\p_z(h^2)=2h\p_zh$. Now since the operator $\p_{\bar{z}}$ is elliptic, from Calder\'on-Zygmund 
estimates by using \eqref{eq:2complex} we find that $(\p_zh)^2$ is in $W^{1,p}_{\text{loc}}(\O,\mathbb{C})$ for 
every $1\leq p<2$. But by Sobolev embeddings we have that $(\p_zh)$ is in $L^q_{\text{loc}}(\O,\mathbb{C})$ for every 
$1\leq q<+\infty$ and so is $|\nabla h|$. Thus we deduce that $h\p_zh$ is in $L^q_{\text{loc}}(\O,\mathbb{C})$ for every 
$1\leq q<+\infty$. By applying elliptic estimates once more we obtain $(\p_zh)^2=(\p_1h)^2-
(\p_2h)^2-2i(\p_1h)(\p_2h)$ is in $W^{1,q}_{\text{loc}}(\O,\mathbb{C})$ for every $1\leq q<+\infty$. From that we also 
have that $|\nabla h|^2=|(\p_zh)^2|$ is in $W^{1,q}_{\text{loc}}(\O)$ for every $1\leq q<+\infty$ and so is 
$T_h$ since its entries are formed by the components of $(\p_zh)^2$ and $h^2$. By using Sobolev 
embeddings we obtain $|\nabla h|^2 \in \C^{0,\a}(\O)$ for every $0<\a<1$ and thus $|\nabla h|\in \C^{0,\beta}(\O)$ 
for every $0<\beta<\frac{1}{2}$. This implies that $h$ is in $W^{1,\infty}_{\text{loc}}(\O)$.
\end{proof}

From Example \ref{ex:example1} we can see that the Lipschitz regularity is optimal. This was obtained by rather direct bootstrap arguments from Equation \eqref{eq:2complex}. To write this complex form of Equation \eqref{2} we used that $h$ is real-valued. If we consider more general inner-variational problems for vector-valued, or complex-valued unknows, Lipschitz regularity can also be obtained but with the help of more elaborated arguments, see \cite{IwaniecKovalevOnninen2013}. We can wonder if Equation \eqref{2} alone implies that $D^2h$ is in $\M(\O)$ and thus \eqref{1}. The following example shows that this is not the case.

\begin{example}\label{ex:example3}
Let $\O=(-1,1)\times (-1,1)$. Let $g$ in $L^1((-1,1),\{\pm 1\})$  be a  function which is not in $BV$, we take $h(x,y)=f(x)=e^{\int_0^xg(s)ds}$, for $(x,y)\in \O$. We can see that $f'(x)=g(x)e^{\int_0^xg(s)ds}$. We thus have $(f')^2=f^2$ and then $h$ satisfies \eqref{2}, but $f''$ in the sense of distributions is not a Radon measure. Indeed if $f'$ were in $BV$ then $g$ would be in $BV$.
\end{example}

Now that we know that $|\nabla h|$ is in $L^\infty_{\text{loc}}(\O)$ we can already restrict the class of measure $\mu$ which can satisfy \eqref{1} and \eqref{2} thanks to the following result of Silhavy, Theorem 3.2 in \cite{Silhavy2005}, and Chen-Torres-Ziemer, Lemma 2.25 in \cite{ChenTorresZiemer2009}:

\begin{proposition}(\cite{Silhavy2005},\cite{ChenTorresZiemer2009})
Let $F$ be in $L^\infty(\O,\R^2)$ such that $\dive F$ is a Radon measure. Then we have that $|\dive F|<< \mathcal{H}^1$, meaning that if $A \subset \R^2$ is such that $\mathcal{H}^1(A)=0$ then $|\dive F|(A)=0$ and thus $\dive F(A)=0$.
\end{proposition}
This result is proved with the help of the notion of $p$-capacity. We do not reproduce the proof here. This directly implies that

\begin{corollary}
Let $h \in H^1(\O)$ be such that \eqref{1} and \eqref{2} are satisfied then $\mu << \mathcal{H}^1$.
\end{corollary}
This means that $\mu$ cannot charge sets of Hausdorff dimension less that $1$. But a priori any Hausdorff dimension between 1 and 2 included are allowed for the components of the support.                     

We also state and prove a boundary regularity result that we need to derive a Pohozaev-type formula. This result is the content of Lemma IV.1 in \cite{SandierSerfaty2003} but we provide the proof for the comfort of the reader.

\begin{proposition}\label{prop:regboundary}
Let $h$ be in $H^1(\O)$ which satisfies \eqref{2} and such that $h$ is constant on $\p \O$ then $T_h$ and $|\nabla h|^2$ are in $W^{1,p}(\O)$ for every $1\leq p<+\infty$. In particular $T_h$ and $|\nabla h|^2$ are well-defined continuous functions on $\p \O$.
\end{proposition}

\begin{proof}
Since $\O$ is simply connected and since the Dirichlet energy is conformally invariant we can use a conformal change of variables and assume that $\O=\R^2_+$ with $h$ that is stationary with respect to inner variations for a new functional $\int_{\R^2_+} \left( |\nabla h|^2 +\phi h^2\right)$. We set $(D_h)_{i,j}:=2\p_ih\p_jh-|\nabla h|^2 \delta_{ij}$ for $1\leq i,j\leq 2$. Then $h$ is a solution of an equation of the form $\dive (D_h)=\dive F$ in $\R^2_+$ where $F=\phi h^2$  and thus $F$ has the same regularity as $h^2$, that is $F$ is in $W^{1,p}(\R^2_+)$ for all $1\leq p<2$. Since we assume that $h$ is constant on the boundary we can extend it to a function on $\R^2$ by setting $\overline{h}(x,y)=h(x,-y)$ for $y<0$ and
\begin{equation}
\tilde{h}(x,y):= \begin{cases}
h(x,y) \text{ if } y\geq 0 \\
\overline{h}(x,y) \text{ if } y<0.
\end{cases}
\end{equation}
We have that $\tilde{h}$ is in $H^1(\R^2)$, thus $D_{\tilde{h}}$ is in $L^1(\R^2)$. Let us compute its divergence and let us show that it does not have any singular part on $\p \R^2_+$. For all $\varphi \in \mathcal{C}^\infty_c(\R^2,\R^2)$ we have
\begin{eqnarray}
\langle \dive(D_{\tilde{h}}),\varphi\rangle &=& -\int_{\R^2} D_{\tilde{h}}\cdot \nabla \varphi \nonumber \\
&=& -\int_{\R^2_+}D_{h}\cdot \nabla \varphi -\int_{\R^2_-} D_{\overline{h}}\cdot \nabla \varphi \nonumber
\end{eqnarray}
Now since $\dive(D_h)=\dive(F)$ is in $L^p(\R^2_+)$ for $1\leq p<2$ we can apply the results of \cite{ChenTorresZiemer2009,Silhavy2005} on a generalization of the Gauss-Green formula and write that
\begin{equation}
\int_{\R^2_+} D_{h}\cdot \nabla \varphi =\int_{\R^2_+}\dive(F)\cdot \varphi +\int_{\p \R^2_+} D_{h}.\varphi \cdot \nu,
\end{equation}
in this formula $D_{h}.\varphi \cdot \nu$ has to be understood in the sense of \cite{ChenTorresZiemer2009} Theorem 5.2 and \cite{Silhavy2005} Proposition 4.2. In the same way we have
\begin{equation}
\int_{\R^2_-} D_{\overline{h}}\cdot \nabla \varphi =\int_{\R^2_-}\dive(\overline{F})\cdot \varphi +\int_{\p \R^2_-} D_{\overline{h}}.\varphi \cdot \nu,
\end{equation}
because $\dive(D_{\tilde{h}})=\dive(\overline{F})$ for some $\overline{F}$ in $W^{1,p}(\R^2_-)$. Now since $D_{h}(x,y)=D_{\tilde{h}}(x,-y)$ for almost every $(x,y)$ in $\R^2_+$  and since the orientations of the outwards normals on $\p \R^2_+$ and $\p \R^2_-$ are opposite, the generalized theory of integration by parts of \cite{ChenTorresZiemer2009,Silhavy2005} allows us to say that 
\begin{equation}
\int_{\p \R^2_+} D_{h}.\varphi \cdot \nu=-\int_{\p \R^2_-} D_{\overline{h}}.\varphi \cdot \nu.
\end{equation}
Hence $\dive(D_{\tilde{h}})$ is in $L^p$ for all $1\leq p<2$ and is equal to
\begin{equation}
\dive(D_{\tilde{h}})=\begin{cases}
\dive(F) \text{ if } y>0 \\
\dive(\overline{F}) \text{ if }y<0.
\end{cases}
\end{equation}
Now we can rewrite this equation in $\R^2$ in a complex form $\p_{\overline{z}}[(\p_z \tilde{h})^2 ]=\p_{z}(G)$ where $G$ has the same regularity as $\tilde{h}^2$. A bootstrap argument as in the proof of Proposition \ref{reg1} allows us to obtain that $|\nabla \tilde{h}|^2$ and $T_{\tilde{h}}$ are in $W^{1,p}_{\text{loc}}(\R^2)$ for all $1\leq p< +\infty$ and thus $|\nabla h|^2$ and $T_h$ are in $W^{1,p}(\R^2_+)$ for all $1\leq p<+\infty$.
\end{proof}

\subsection{Pohozaev formula and a non-existence result}

Pohozaev-type formulas are well-known tools for proving non-existence results in elliptic partial differential equations, cf.\ e.g.\ \cite{Pohozaev1965}. For many variational problems they can be obtained by multiplying the Euler-Lagrange equations by $(x\cdot \nabla u)$, where $u$ is the unknown of the problem, and integrating on the domain. But it is also known that they can be obtained by using only the inner-variational equations, see e.g.\ \cite{AlikakosFaliagas2012}. In order to do that we take the inner product of $x$, the position vector, and $\dive (T)$ where $T$ is the stress-energy tensor associated to the problem, and we integrate on $\O$. Since we do have an inner-variational equation in our problem a Pohozaev formula holds.

\begin{proposition}
Let $h$ be in $H^1(\O)$ such that $\eqref{2}$ holds, then we have
\begin{equation}\label{eq:Pohozaev}
\frac{1}{2}\int_{\p \O} (x\cdot \nu)\left(|\p_\tau h|^2-|\p_\nu h|^2+h^2\right)-\int_{\p \O} (x\cdot \tau) \p_\tau h \cdot \p_\nu h -\int_\O h^2=0,
\end{equation}
where $\nu$ denotes the outward normal to $\O$ and $\tau$ the direct tangent vector to $\p \O$.
\end{proposition}
Note that since $h$ satisfies \eqref{2} and $h$ is constant on $\p \O$ we have that $T_{h}$ and $|\nabla h|^2$ are in $W^{1,p}(\O)$ for every $1\leq p<+\infty$, cf.\ Proposition \ref{prop:regboundary}. Thus $T_h$ has a well-defined trace on $\p \O$.
\begin{proof}
In the following, repeated indices induce a summation process. We have $\dive(T_h)=0$ in $\O$ and thus $\int_\O x\cdot \dive(T_h)=0$. This reads
\begin{equation}
\int_\O x_j\p_i\left[ \p_ih \p_jh- \frac{1}{2}(|\nabla h|^2+h^2)\delta_{ij}\right]=0.
\end{equation}
We can integrate by parts, since $T_h$ is in $W^{1,p}(\O)$ for all $1\leq p<+\infty$ and this yields
\begin{eqnarray}
-\int_\O \delta_{ij} \left[ \p_ih \p_jh-\frac{1}{2}(|\nabla h|^2+h^2)\delta_{ij}\right] \phantom{aaaaaaaaaaaaaaaaaa}   \nonumber \\  
+\int_{\p \O} x_j\left[ \p_ih\p_jh-\frac{1}{2}(|\nabla h|^2+h^2)\delta_{ij} \right]\nu_i=0 \\ \nonumber
\Rightarrow \int_\O h^2+\int_{\p \O} (x\cdot \nabla)h \cdot \p_\nu h-\int_{ \p \O} (x\cdot \nu) \frac{1}{2}(|\nabla h|^2+h^2)=0.
\end{eqnarray}
By using that $x=(x\cdot \nu) \nu+(x\cdot \tau) \tau$ we find
\begin{eqnarray}\nonumber
\int_\O h^2+\int_{\p \O} (x\cdot \nu) |\p_\nu h|^2+(x\cdot \tau) \p_\nu h\p_\tau h-\int_{\p \O} (x\cdot \nu)\frac{1}{2}(|\nabla h|^2+h^2)=0,
\end{eqnarray}
and this concludes the proof.
\end{proof}

Classically a Pohozaev formula implies a non-existence result when we work in star-shaped domains and with an appropriate boundary condition. This is the content of Theorem \ref{th:main2}.

\begin{proof}(Proof of Theorem \ref{th:main2})
Without loss of generality we can assume that $\O$ is star-shaped with respect to the origin. This means that $(x\cdot \nu)\geq 0$ for all $x$ on $\p \O$. We then have
\begin{equation}
-\frac{(x\cdot \nu )}{2}\int_{\p \O} |\p_\nu h|^2 =\int_\O h^2.
\end{equation}
This implies that $\int_\O h^2=0$ and thus $h=0$ in $\O$.
\end{proof}

We do not know if this non-existence result remains true in every simply connected domain. But if we consider some domain with non trivial topology then a non trivial solution of \eqref{1}-\eqref{2} with $h=0$ on $\p \O$ may exist as shown by the following example:

\begin{example}\label{ex:circle}
Let $R>2>r$ and $\O:=B_R(0) \setminus \overline{B_r(0)}$. We consider the solutions of 
\begin{equation}\label{eq:h_1}
\left\{
\begin{array}{rcll}
-\Delta h_1+h_1&=&0 \ & \text{ in } B_2(0)\setminus \overline{B_r(0)} \\
h_1&=&0 \ & \text{ on } \p B_r(0) \\
h_1&=& 1 \ & \text{ on } \p B_2(0).
\end{array}
\right.
\end{equation}
and 
\begin{equation}\label{eq:h_2}
\left\{
\begin{array}{rcll}
-\Delta h_2+h_2&=&0 \ & \text{ in } B_R(0)\setminus \overline{B_2(0)} \\
h_2&=&1 \ & \text{ on } \p B_2(0) \\
h_2&=& 0 \ & \text{ on } \p B_R(0).
\end{array}
\right.
\end{equation}
Both $h_1$ and $h_2$ are radially symmetric. We claim that there exists an $R>2>r$ such that $\p_\rho h_1(2)=-\p_\rho h_2(2)$. Then taking $h=\begin{cases} h_1 \text{ in } B(0,2) \setminus \overline{B(0,r)} \\
h_2 \text{ in } B(0,R) \setminus \overline{B(0,2)} \end{cases}$ we can check that $h$ satisfies $h=0$ on $\p \O$ along with \eqref{1} and \eqref{2} where $\mu$ is a non-trivial measure supported on the circle $\rho=2$, more precisely $\mu=-2\p h_1/\p \rho (2,\theta)\mathcal{H}^1_{\lfloor\{\rho=2\}}.$
\end{example}

\begin{proof}
We use the \textit{modified Bessel functions} as in \cite{Aydi2008} and \cite{Le2009}. Let $I_0$ and $K_0$ be respectively the modified Bessel function of the first kind and of the second kind:
\begin{equation}
I_0(x)=\sum_{n=0}^{+\infty}\frac{x^{2n}}{(n!)^2 2^{2n}}, \ \ \ K_0(x)=-(\log (x/2)+\gamma)I_0(x)+\sum_{n=0}^{+\infty} \frac{x^{2n}}{(n!)^2 2^{2n}}\phi(n),
\end{equation}
where $\phi(n)=\sum_{k=1}^n 1/k$ for $n\neq 0$, $\varphi(0)=0$, and $\gamma=\lim_{n\rightarrow +\infty}(\phi(n)-\log n)$.
\end{proof}
These are solutions of 
\begin{equation}
-y''-\frac{y'}{x}+y=0 \ \ \ \text{ for } 0\leq x <+\infty
\end{equation}
with a singularity at $0$ for $K_0$. We refer to \cite{Watson} for more on Bessel functions. We set $I_1(x):=I_0'(x)$ and $K_1(x):=-K_0'(x)$ for $0<x<+\infty$. We can express the solutions $h_1$ and $h_2$ in terms of these Bessel functions. Indeed
\begin{equation}
h_1(\rho)=\a I_0(\rho)+\beta K_0(\rho), \text{ for } r\leq \rho \leq 2,
\end{equation} 
\begin{equation}
h_2(\rho)=\gamma I_0(\rho)+ \delta K_0(\rho), \text{ for } 2\leq \rho \leq R,
\end{equation}
with
\begin{eqnarray}
 \a = \frac{-K_0(r)}{K_0(2)I_0(r)-I_0(2)K_0(r)}, \ \ \ \ \beta =\frac{I_0(r)}{K_0(2)I_0(r)-I_0(2)K_0(r)}, \\
 \gamma=\frac{-K_0(R)}{K_0(2)I_0(R)-I_0(2)K_0(R)}, \ \ \ \ \delta=\frac{I_0(R)}{K_0(2)I_0(R)-I_0(2)K_0(R)}.
\end{eqnarray}
We have that $h_1'(2)=:f(r)$ and $h_2'(2)=:g(R)$ with $f(r)\rightarrow +\infty$ as $r\rightarrow 2$, $f(r)\rightarrow \frac{I_1(2)}{I_0(2)}$ as $r\rightarrow 0$ and $g(R)\rightarrow-\infty$ as $R \rightarrow 2$ and $g(R) \rightarrow-\frac{K_1(2)}{K_0(2)}$ as $R\rightarrow +\infty$. Besides $f,g$ are continuous. Thus, by choosing $0<r<2$ such that $f(r)> \frac{K_1(2)}{K_0(2)}$ we can choose an $R>2$ such that $f(r)=-g(R)$. This shows that we can choose $r<2<R$ such that $h_1,h_2$ solutions of \eqref{eq:h_1}, \eqref{eq:h_2} satisfy $\p_\rho h_1(2)=-\p_\rho h_2(2)$. Now by defining $h$ as $h=h_1$ in $B_2(0)\setminus B_r(0)$  and $h=h_2$ in $B_R(0) \setminus B_2(0)$ we can check that $h$ is in $H^1(\O)$ and $h$ satsifies \eqref{2} and \eqref{1}.
\section{Local regularity near regular points of $h$}\label{II}

In this section we describe the regularity of $h\in H^1(\O)$ and $\mu \in \M(\O)$ which satisfy \eqref{1} and \eqref{2} near a point $z_0 \in \supp \mu$ such that $\nabla h(z_0) \neq0$. Note that the last condition is meaningful since $|\nabla h|\in \C^0(\O)$. Our main result is that near a point of the support of $\mu$ that is also a regular point of $h$,  $\mu$ is supported by a regular $\C^1$ curve.

\begin{theorem}\label{th:main4}
Let $h\in H^1(\O)$ and $\mu \in \M(\O)$ be such that \eqref{1} and \eqref{2} are satisfied. Let $z_0\in \supp \mu$ be such that $|\nabla h(z_0)| \neq 0$. Then there exist $R>0$, $\theta \in BV(B_R(z_0),\{\pm 1\})$ and $H\in \C^{1,\alpha}(B_R(z_0))$ for every $0<\alpha<1$ such that $\nabla h=\theta \nabla H$ in $B_R(z_0)$
\begin{equation}
\supp \mu_{\lfloor B_R(z_0)}=\{x\in B_R(z_0) ; H(x)=0\}=:\Gamma.
\end{equation}
Furthermore $\Gamma$ is a regular $\C^{1,\alpha}$ curve for every $0<\alpha<1$, it is part of the boundary of $\{\theta=1 \}$ and, if we denote by $\nu$ the exterior normal to $\{ \theta=1\}$ restricted on $\Gamma$ we have
\begin{equation}
\mu_{\lfloor \{B_r(z_0) \}}=2\p_\nu H \mathcal{H}^1_{\lfloor \Gamma}.
\end{equation}
We also have
\begin{equation}\label{eq:localsignofmu}
\mu_{\lfloor B_R(z_0)}=+2|\nabla h|  \H^1_{\lfloor \Gamma} \ \  \text{ or } \ \ \ \mu_{\lfloor B_R(z_0)}=-2|\nabla h|  \H^1_{\lfloor \Gamma}.
\end{equation}
Besides we have that $\nabla h$ is in $BV(B_R(z_0))$ and 
\begin{equation}\label{eq:BVVV}
\|\p_{ij}h\|_{\M(B_R(z_0))}\leq |\mu|(B_R(z_0))+C \|h\|_{L^\infty(B_R(z_0))}|B_R(z_0) |
\end{equation}
where $C>0$ is a constant which does not depend on $h$, $z_0$ nor on $R$.
\end{theorem}

The situation depicted in this result is that of Example \ref{ex:example1} and Example \ref{ex:circle}. This theorem also shows that $\mu$ is locally a Radon measure with fixed sign near the regular points of $h$. This generalizes a result of Le, see Corollary 1.1 in \cite{Le2009}. We note that the regularity of the function $H$ appearing in the theorem is almost optimal as shown by Example \eqref{ex:example1}. Indeed in this example we can write $h= \theta H$ in $\O$ with $\theta=1$ in $(-1,0)\times(-1,1)$, $\theta=-1$ in $(0,1)\times (-1,1)$ and $H(x,y)=\tilde{f}(x)=\int_0^x \tilde{g(s)}ds$ where $\tilde{g}(x)=\begin{cases} e^x & \text{ if } x<0,\\ e^{-x} & \text{ if } x>0 \end{cases}$. We can see that $H$ is in $\mathcal{C}^{1,1}(\O)$ but is not in $\mathcal{C}^2(\O)$.

In order to prepare the proof of this theorem we first state and prove several lemmas.

\begin{lemma}\label{lemme1}
Let $h\in H^1(\O)$ be a solution of \eqref{eq:2complex}, there exist $\alpha \in W^{1,q}_{\text{loc}}(\O,\mathbb{C})$ for every $1\leq q<+\infty$ and $f$ a holomorphic function in $\O$, such that
\begin{equation}
\p_{\overline{z}}\alpha = \frac{1}{4}\p_z(h^2) \ \ \text{ in } \O
\end{equation}
\begin{equation}
(\p_zh)^2=\a +f \ \ \text{ in } \O.
\end{equation}
\end{lemma}

\begin{proof}
We look for a particular solution of $\p_{\overline{z}} \alpha= \p_z(h^2)$. We can proceed as in \cite{BethuelBrezisHelein1994} p.66-67. We let $T=\p_{z}(\frac{1}{\pi z})=\p^2_{zz} (\frac{2}{\pi}\ln r)$ and $\tilde{h^2}:=\begin{cases}
h^2 & \text{ in } \O \\
0 & \text{ in } \R^2 \setminus \O
\end{cases}$.
We then set $\a= T\ast \tilde{h^2}$. We can check that $\p_{\overline{z}} \alpha =\p_{z} (\tilde{h^2})$ in $\D'(\R^2)$. This is because $\p_{\overline{z}}(\frac{1}{\pi z})=\p_{\overline{z}} \p_{z} (\frac{2}{\pi}\ln r)=\frac{1}{2\pi}\Delta \ln r= \delta_0$. Thus we get
\[ \p_{\overline{z}}\left[ (\p_zh)^2-\alpha \right]=0 \ \ \text{ in } \D'(\O).\]
Since the operator $\p_{\overline{z}}$ is elliptic, from the elliptic regularity theory, we can deduce that there exists $f$ holomorphic in $\O$ such that
\[ (\p_zh)^2=\a +f \ \ \text{ in } \O .\]
The fact that $\a$ belongs to $W^{1,q}_{\text{loc}}(\O)$ for every $1\leq q<+\infty$ is due to the regularity of $(\p_zh)^2$ established in Lemma \ref{reg1}.
\end{proof}

\begin{lemma}\label{lemme2}
Let $h\in H^1(\O)$ and $\mu \in \M(\O)$ which satisfy \eqref{1}, \eqref{eq:2complex}. Let $\a$ and $f$ be as in Lemma \ref{lemme1} such that $(\p_zh)^2=\a+f$. Let $z_0 \in \supp \mu$ be such that $\nabla h(z_0)\neq 0$. Then there exist $R>0$ and $g\in W^{1,q}(B_R(z_0),\mathbb{C})$ for every $1\leq q<+\infty$ such that $g^2= \a+f$. Besides there exists $\theta: B_R(z_0)\rightarrow \{ \pm 1 \}$ such that
\begin{equation}\label{eq:deriveeh}
\p_zh=\theta g,
\end{equation}
\begin{equation}\label{eq:deriveeg}
\p_{\overline{z}}g= \frac14 \theta h
\end{equation}
and $\theta$ is in $BV(B_R(z_0),\{\pm 1\})$.
\end{lemma}

\begin{proof}
Let $z_0 \in \supp \mu$ be such that $\p_zh(z_0) \neq 0$. Since $(\p_zh)^2=\a+f$ is continuous in $\O$, we can find $R>0$ such that $\p_zh$ does not vanish in $B_R(z_0)$. We can then define a complex square root of $\a+f$ in $B_R(z_0)$. More precisely there exists $g:B_R \rightarrow \mathbb{C}$ such that $g^2=\a+f=(\p_zh)^2$. To construct such a $g$ we can take $g=e^{\frac{1}{2}\log (\a +f)}$, where $\log (\a +f)$ can be defined as in Lemma \ref{lem:A1}. Furthermore, from Lemma \ref{lem:A1} and the chain rule in Sobolev spaces, we also have that $g\in W^{1,q}(B_R(z_0),\mathbb{C})$ for every $1\leq q<+\infty$. We then have the existence of $\theta:B_R(z_0) \rightarrow \{\pm 1\}$ such that \eqref{eq:deriveeh} holds.
Besides since $g\in W^{1,q}(B_R(z_0))\cap \C^0(\overline{B_R(z_0)})$ for every $1\leq q<+\infty$ we can infer that $\p_{\bar{z}}(g^2)=2g\p_{\bar{z}}g$. On the other hand by using \eqref{eq:2complex} we have $\p_{\bar{z}}\left[ (\p_zh)^2\right]=\frac{1}{4}\p_z(h^2)=\frac{1}{2}h\p_zh$. Thus $g\p_{\bar{z}}g=\frac{1}{4}h\p_zh$ and, since $g=\theta \p_zh$, we find $\p_{\bar{z}}g=\frac{1}{4}\theta h.$

We now show  that $\theta$ is in $BV(B_R(z_0),\{\pm 1\})$. Since $g$ does not vanish in 
$B_R(z_0)$, we can write $\theta=\frac{\p_zh}{g}$, which proves that $\theta$ is measurable 
and since $\theta$ is in $L^\infty(B_R(z_0))$ we have that $\theta \in L^1(B_R(z_0))$. Let $
\ph \in \C^\infty_c(B_R(z_0))$ be a test function. We compute
\begin{eqnarray}
\langle 2\p_{\overline{z}} \theta,\ph \rangle &=& \langle \p_x\theta +i\p_y\theta, \ph \rangle = -\langle \theta,\p_x \ph+i\p_y\ph \rangle \nonumber \\
&=& -\int_{B_R(z_0)}\frac{(\p_xh-i\p_yh)}{2g}(\p_x\ph+i\p_y \ph) \nonumber \\
&=& -\int_{B_R(z_0)} \frac{\p_xh\p_x\ph+\p_yh\p_y\ph+i(\p_xh\p_y\ph-\p_yh\p_x\ph)}{2g}. \nonumber
\end{eqnarray}
Now we use the fact that since $g$ does not vanish and since $g\in W^{1,q}$ for every $1\leq q<+\infty$, we have
\[ \p_j \left( \frac{\varphi}{g} \right)=\frac{\p_j\ph}{g}-\frac{\ph \p_j g}{g^2} \ \ \text{ for } j=1,2. \]
Thus
\begin{eqnarray}
\langle 2\p_{\overline{z}} \theta, \ph \rangle &=& -\frac12\int_{B_R(z_0)} \p_xh \left[ \p_x(\frac{\ph}{g})+\frac{\ph \p_xg}{g^2}\right] +\p_yh\left[\p_y(\frac{\ph}{g})+\frac{\ph\p_yg}{g^2}\right] \nonumber \\
& & \phantom{aaaaaa } - \frac{i}{2} \int_{B_R(z_0)}  \p_xh \left[ \p_y (\frac{\ph}{g})+\frac{\ph \p_yg}{g^2}\right]-\p_yh\left[ \p_x (\frac{\ph}{g})+\frac{\p_xg \ph}{g^2} \right] \nonumber \\
&=& \langle \frac{\Delta h}{2g},\ph \rangle- \frac12\int_{B_R(z_0)}\left( \frac{\p_xh \p_xg}{g^2}+\frac{\p_yh\p_yg}{g^2}\right)\varphi \nonumber \\
& & \phantom{aaaaaa }-\frac{i}{2} \int_{B_R(z_0)} \left( \frac{\p_xh \p_yg}{g^2}+\frac{\p_yh\p_xg}{g^2}\right)\varphi \nonumber
\end{eqnarray}
In the last equality we used that $\langle \frac{\Delta h}{g}, \varphi \rangle$ is well-defined since $\Delta h$ is a Radon measure and $1/g$ is a continuous function. Besides, since $h$ is in $H^1(\O)$ and $\p_i (\varphi/g)$ is in $L^2(B_R(z_0),\mathbb{C})$ we do have
\begin{equation}
\langle \frac{\Delta h}{g},\varphi \rangle =\int_{B_R(z_0)}\p_xh \p_x(\frac{\varphi}{g})+\p_yh \p_y(\frac{\varphi}{g}).
\end{equation}
We also used that 
\begin{equation}
\int_\O \p_xh \p_y(\frac{\varphi}{g})=\int_\O \p_yh \p_x (\frac{\varphi}{g}).
\end{equation}
This can be seen by approximating $\varphi/g$ by smooth functions in $W^{1,q}(B_R(z_0),\mathbb{C})$. We thus find
 \begin{eqnarray}
\langle 2\p_{\overline{z}} \theta, \ph \rangle &=& \langle \frac{\Delta h}{2g},\ph\rangle -\frac12\int_{B_R(z_0)} \p_xh \left( \frac{\p_xg+i\p_yg}{g^2}\right)\ph  \nonumber \\
& & \phantom{aaaaaa } -\frac{i}{2}\int_{B_R(z_0)}\p_yh\left(\frac{\p_xg+i\p_yg}{g^2}\right) \nonumber \\
&=&\langle \frac{\Delta h}{2g},\ph\rangle -\int_{B_R(z_0)} \frac{(\p_xh-i\p_yh)(\p_{\overline{z}}g) \varphi}{g^2}. \nonumber
\end{eqnarray}
Now we recall that $g^2=(\p_zh)^2$, thus $\p_{\overline{z}}g=\frac{1}{4}\frac{h\p_zh}{g}$ and 
\begin{eqnarray}
\int_{B_R(z_0)} \frac{(\p_xh-i\p_yh)(\p_{\overline{z}}g) \varphi}{g^2}=\int_{B_R(z_0)} \frac{1}{2} \frac{h(\p_zh)^2}{g^2 g} =\int_{B_R(z_0)} \frac 12 \frac{h}{g}. \nonumber
\end{eqnarray}
Taking into account that $\Delta h= h-\mu$ and that $\mu/g$ is well-defined because $g$ is continuous and $\mu$ is a Radon measure we obtain
\[ 2\langle \p_{\overline{z}}\theta , \ph\rangle =\langle \frac{-\mu}{2g},\varphi \rangle. \]
That is 
\begin{equation}\label{eq:deriveesdetheta}
\p_x\theta =-\text{Re}\left(\frac{1}{2g}\right)\mu \ \  \text{ and } \ \ \ \p_y\theta =-\text{Im}\left(\frac{1}{2g}\right)\mu.
\end{equation}
The right-hand sides in the last equation are Radon measures and hence $\theta \in BV(B_R(z_0))$. \\

\end{proof}

\begin{lemma}\label{lem:BVestimate}
Let $h\in H^1(\O)$ and $\mu \in \mathcal{M}(\O)$ be such that \eqref{1} and \eqref{eq:2complex} hold. Let $z_0$ be such that $\nabla h(z_0)\neq 0$ then there exists $R>0$ such that $\nabla h \in BV(B_R(z_0),\R^2)$ and 
\begin{equation}\label{eq:BVestimate}
|\p_{ij}^2 h|(B_R(z_0))\leq |\mu|(B_R(z_0))+C\|h\|_{L^\infty(B_R(z_0))}|B_R(z_0)|,
\end{equation}
for some $C>0$ which does not depend on $z_0$ nor on $R$.
\end{lemma}

\begin{proof}
From Lemma \ref{lemme2} we can find $\tilde{R}>0$, $\theta\in BV(B_{\tilde{R}}(z_0))$ and $g
\in W^{1,p}(B_{\tilde{R}}(z_0),\mathbb{C})$ such that $\p_zh=\theta g=\theta(g_1+ig_2)$ and $
\p_{\bar{z}}g=\theta h$ in $B_{\tilde{R}}(z_0)$. We also have from \eqref{eq:deriveesdetheta} 
that $\p_x\theta =-\RE \left(\frac{1}{2g}\right)\mu$ and $\p_y \theta=\IM  \left(\frac{1}
{2g}\right)\mu$ in $B_{\tilde{R}}(z_0)$. We can apply Lemma \ref{lem:A2} to obtain that $
\p_xh=\theta g_1$ and $\p_yh=-\theta g_2$ are in $BV(B_{\tilde{R}}(z_0))$ and 
\begin{equation}\label{eq:previous}
\p_{i1}^2h=\p_i\theta g_1+\theta\p_ig_1, \ \ \ \ \ \  \p_{i2}h=-(\p_i\theta g_2+\theta\p_ig_2)
\end{equation}
for $1\leq i \leq 2$. To establish \eqref{eq:BVestimate} we set $R:=\tilde{R}/2$.  Now for  $\varphi \in \C^\infty_c(B_R(z_0))$ such that $\|\varphi\|_{L^\infty}\leq 1$ we have 
\begin{eqnarray}
\left| \int_{B_R(z_0)} \p_ih \p_1 \varphi \right| \leq  \left|\int_{B_R(z_0)} \RE (\frac{1}{2g})g_1\varphi d\mu \right| +\left|\int_{B_R(z_0)}   \theta \p_1 g_i\varphi \right| \nonumber \\
\leq  |\mu| (B_R(z_0))+ \|\nabla g\|_{L^p(B_R(z_0))}|B_R(z_0)|^{1/p'}, \nonumber
\end{eqnarray}
where we have used that $ |g_1\RE (\frac{1}{g})|\leq 1$. Since $\p_{\bar{z}}g=\theta h$, by elliptic estimates, for every $1<p<+\infty$ we have $\|\nabla g\|_{L^p(B_R(z_0))}\leq C_p\|h\|_{L^p(B_R(z_0))}$. By scaling $C_p$ does not depend on $R$. By translation invariance of $|\p_{\bar{z}}g|=|h|$, $C$ does not depend on $z_0$. Indeed let $\tilde{g}(z):= g(z_0+2Rz)$, $\tilde{\theta}(z):=\theta(z_0+2Rz)$ and $\tilde{h}(z):=h(z_0+2Rz)$. Then we have that 
\begin{equation}
\p_{\bar{z}}\tilde{g}=2R\tilde{\theta}\tilde{h} \text{  in } B_1(0).
\end{equation}
Thus $\| \nabla \tilde{g}\|_{L^p(B_{1/2}(0))} \leq C_p R \|\tilde{h} \|_{L^p(B_{1/2}(0))}$, for $C_p$ which only depends on $1<p<+\infty$. But with a change of variable we obtain 
\begin{equation}
\| \nabla g\|_{L^p(B_{R}(0))} \leq C_p \|h \|_{L^p(B_{R}(0))}.
\end{equation}

  Since $\|h\|_{L^p(B_R(z_0))}\leq \|h\|_{L^\infty(B_R(z_0))}|B_R(z_0)|^{1/p}$ , and since the computations are similar for the other second derivatives of $h$ we can conclude the lemma by taking $C=C_2$ for example.
\end{proof}

We now set
\[ B_R^+:=\{ z\in B_R(z_0); \theta(z)=1 \} \ \ \ \ \text{ and } B_R^-:=\{z\in B_R(z_0); \theta(z)=-1 \}. \]
We also set $g=g_1+ig_2$. Since $\theta$ is in $BV(B_R(z_0),\{\pm 1\})$ we have that $B_R^+$ and $B_R^-$ are sets of finite perimeter in $B_R$. The next lemma gives a geometric information on the generalized normal to the essential boundary of $B_R^+$. For definitions and results about sets of finite perimeter we refer to \cite{AmbrosioFuscoPallara2000,EvansGariepy2015}.

\begin{lemma}\label{lemme3}
Let $h\in H^1(\O)$ and $\mu\in \M(\O)$ which satisfy \eqref{1} and \eqref{2}. Let $z_0\in \supp \mu$ be such that $\nabla h(z_0)\neq 0$. Let $R>0$, $\theta:B_R(z_0) \rightarrow \{ \pm 1\}$, $g=g_1+ig_2 \in W^{1,q}(B_R(z_0),\mathbb{C})$ for every $1\leq q<+\infty$ be as in Lemma \ref{lemme2}. Let $B_R^+:=\{ z\in B_R(z_0); \theta(z)=1 \}$, this is a set of finite perimeter and the generalized normal to the essential boundary of $B_R^+$ denoted by $\nu_{B_R^+}$ satisfies that 
\begin{equation}
\nu_{B_R^+}(x)=\lambda(x) (-g_1(x),g_2(x)),
\end{equation}
for $\mathcal{H}^1$-almost every $x$ in $\p_\star B_R^+\setminus \p B_R(z_0)$ and for $\lambda: \p_\star B_R^+\setminus \p B_R(z_0)\rightarrow \R$.
\end{lemma}

\begin{proof}
From the definition of $g$ and $\theta$ we have that
\[ \p_xh-i\p_yh=2\theta g= 2\theta (g_1+ig_2). \]
We must then have that
\begin{equation}\label{equality}
 \p_y (\theta g_1)=-\p_x(\theta g_2)
\end{equation}
in the sense of distributions. Since $\theta \in BV(B_R(z_0))$ and $g\in \C^0(\overline{B_R(z_0)})\cap W^{1,q}(B_R,\mathbb{C})$ for every $1\leq q<+\infty$, we can use the Leibniz formula, cf.\ Lemma \ref{lem:A2}, to obtain
\[ \p_y (\theta g_1)= \p_y\theta g_1+ \theta \p_yg_1 \text{ and }  \p_x (\theta g_2)= \p_x \theta g_2+\theta \p_y g_2. \] 
From \eqref{equality}, we then have
\[\p_y\theta g_1+ \theta \p_yg_1 = -\p_x \theta g_2-\theta \p_y g_2.\]
That is
\begin{eqnarray}
\p_y\theta g_1 +\p_x\theta g_2&=&-\theta ( \p_yg_1+\p_xg_2)  = -\theta \text{Im} (\p_{\overline{z}}g) \nonumber \\
&=& -\theta \text{Im} (\frac{\theta h}{4}) = 0. \nonumber
\end{eqnarray}
In the last equality we have used \eqref{eq:deriveeg}.

We deduce that $\p_y \theta g_1 +\p_x \theta g_2=0$ in $\D'(B_R(z_0))$. We multiply by a test 
function $\varphi \in \C^\infty_c(B_R(z_0))$, and we use that $\p_x\theta, \p_y\theta$ are 
Radon measures and $g$ is continuous to obtain
\[\langle \p_y \theta,g_1\ph \rangle +\langle \p_x\theta,g_2\ph \rangle=0.\]
We let $\psi:=(g_2\ph,g_1\ph)$ and we can see that the previous equality can be rewritten as
\begin{eqnarray}
\int_{B_R(z_0)} \theta \dive \psi =\int_{B_R^+}\dive \psi -\int_{B_R^-} \dive \psi=0.  \nonumber
\end{eqnarray}
We now use the Gauss-Green formula for sets of finite perimeter, cf.\ \cite{EvansGariepy2015}, and we find that
\[2\int_{\p_\star B_R^+} \psi\cdot \nu_{B_R^+}d\H^1 =0. \]
Note that the Gauss-Green formula is valid for vector fields in $\C^1(\overline{B_R(z_0)},\R^2)$, but since $\psi$ is in $W^{1,p}(B_R(z_0),\R^2)$ for every $1\leq p<+\infty$, we can use an approximation argument to show that it is also valid for our $\psi$. Recall that $\psi=(g_2\varphi,g_1\varphi)$ and that the last equality is true for all $\varphi\in \C^\infty_c(B_R(z_0))$, hence we obtain $g_2\nu_1+g_1\nu_2=0$ $\H^1$-almost everywhere on $\p_\star B_R^+ \setminus \p B_R$. This means precisely that $\nu_{B_R^+}$ is collinear to $(-g_1,g_2)$ $\H^1$-almost everywhere on $\p_\star B_R^+\setminus \p B_R(z_0)$.
\end{proof}

We are now in position to give a first description of the support of $\mu$.
\begin{lemma}\label{lemme4}
Let $h\in H^1(\O)$, $\mu \in \M(\O)$ be such that \eqref{1}, \eqref{2} hold. Let $R>0$, $\theta \in BV(B_R(z_0))$ and $g=g_1+ig_2 \in W^{1,q}(B_R(z_0),\mathbb{C})$ for every $1\leq q<+\infty$ be as in Lemma \ref{lemme2}. Let $B_R^+=\{z\in B_R;\theta(z)=+1\}$. We have
\begin{equation}\label{eq:supp}
\supp \mu = \p_\star B_R^+ \setminus \p B_R(z_0),
\end{equation}
\begin{equation}\label{eq:description1}
\mu_{\lfloor B_R}= 4(g_1\nu_1-g_2\nu_2)\H^1_{\lfloor \p_\star B_R^+ \setminus \p B_R(z_0)},
\end{equation}
where $\nu_{B_R^+}=(\nu_1,\nu_2)$ is the generalized normal to $\p_\star B_R^+ \setminus \p B_R(z_0)$.
\end{lemma}

\begin{proof}
Let $\ph \in \C^\infty_c(B_R(z_0))$ be a test function, we have
\begin{eqnarray}
\langle \mu, \ph \rangle & = &\langle -\Delta h+h,\ph \rangle= \int_{B_R(z_0)} \nabla h \cdot \nabla \ph +h\ph \nonumber \\
&=& 2\int_{B_R^+} g_1\p_x\ph -g_2\p_y\ph+2\int_{B_R^-} (-g_1\p_x\ph+g_2\p_y\ph)+\int_{B_R} h\ph \nonumber \\
&=& 2\int_{B_R^+} (-\p_xg_1 +\p_yg_2)\ph +2\int_{B_R^+}(\p_xg_1-\p_yg_2)\ph +\int_{B_R} h\ph  \nonumber \\
& & \phantom{aaaaaa } + 4\int_{\p_\star B_R^+ \setminus \p B_R} (g_1\ph \nu_1-g_2\ph \nu_2)d\H^1 \nonumber \\
&=& -4\int_{B_R^+} \RE(\p_{\overline{z}}g)\ph+4\int_{B_R^-} \RE(\p_{\overline{z}}g)\ph+\int_{B_R} h\ph \nonumber \\ 
& & \phantom{aaaaaaa }+4\int_{\p_\star B_R^+\setminus \p B_R}(g_1\ph\nu_1-g_2\ph\nu_2)d\H^1. \nonumber
\end{eqnarray}
Now recall that $\p_{\overline{z}}g=\frac{1}{4}\frac{h\p_zh}{g}=\frac{\theta h}{4}$ and thus 
\begin{eqnarray}
\langle \mu, \ph \rangle & = &-\int_{B_R^+}h\ph-\int_{B_R^-} h\ph +\int_{B_R} h\ph +4\int_{\p_\star B_R^+\setminus \p B_R}(g_1\ph\nu_1-g_2\ph\nu_2)d\H^1 \nonumber \\
&=& 4\int_{\p_\star B_R^+ \setminus B_R}(g_1\ph\nu_1-g_2\ph\nu_2)d\H^1.
\end{eqnarray}
But since from Lemma \ref{lemme3}, $\nu_{B_R^+}=(\nu_1,\nu_2)$ is collinear to $(-g_1,g_2)$ and $g$ does not vanish in $B_R(z_0)$ we obtain \eqref{eq:supp} and \eqref{eq:description1}.
\end{proof}

Our next step is to push further the analysis of $\p_\star B_R^+ \setminus \p B_R(z_0)$. We want to show that, up to take a smaller $R>0$, $\p_\star B_R^+ \setminus \p B_R(z_0)$ is a regular $\C^{1,\a}$ curve (for all $0<\alpha<1$). We begin by using a result of Ambrosio-Caselles-Morel-Masnou which allows us to decompose $\p_\star B_R^+$ in a countable union of connected simple rectifiable closed curves.

In order to state the next theorem, following \cite{AmbrosioCasellesMasnouMorel2001},  we introduce a \textit{formal} Jordan curve $J_\infty$ whose interior is
$\R ^n$ and a \textit{formal} Jordan curve $J_0$ whose interior is empty. We denote by $\mathcal{S}$ the set of
Jordan curves and formal Jordan curves. We then have the following description of the essential boundary of sets
of finite perimeter in $\R^2$.

\begin{theorem}[Corollary 1 \cite{AmbrosioCasellesMasnouMorel2001}]\label{ACMM}
Let $E$ be a subset of $\R^2$ of finite perimeter. Then there is a unique decomposition of $\partial_\star E$ into rectifiable Jordan curves $\{C_i^+, C_k^-: i,k \in \mathbb{N} \}\subset \mathcal{S}$, such that

\begin{itemize}
\item[i)]Given $\inte(C_i^+), \inte(C_k^+), i\neq k$, they are either disjoint or one is contained in the other; given  $\inte(C_i^-), \inte(C_k^-), i\neq k$, they are either disjoint or one is contained in the other. Each $\inte(C_i^-)$ is contained in one of the $\inte(C_k^+)$.
\item[ii)]$P(E)=\sum_i \mathcal{H}^1(C_i^+) +\sum_k \mathcal{H}^1(C_k^-)$.
\item[iii)] If $\inte(C_i^+) \subset \inte(C_j^+)$, $i \neq j$, then there is some rectifiable Jordan curve $C_k^-$ such that $\inte(C_i^+) \subset \inte (C_k^-) \subset \inte(C_j^+)$. Similarly if $\inte(C_i^-) \subset \inte(C_j^-)$, $i \neq j$, then there is some rectifiable Jordan curve $C_k^+$ such that $\inte(C_i^-) \subset \inte (C_k^+) \subset \inte(C_j^-)$.
\item[iv)] Setting $L_j=\{i ; \ \inte(C_i^- \subseteq \inte(C_j^+)\}$, the sets $Y_j=\inte(C_j^+) \setminus \cup_{i\in L_j} \inte(C_i^-)$ are pairwise disjoint, indecomposable and $E=\cup_{j} Y_j$.
\end{itemize}
\end{theorem}

With the help of the previous theorem, by showing that $B_R^+$ is a set of finite perimeter in $\R^2$ and by using a suitable version of the coarea formula, we can proceed as in Lemma 3.15 and Lemma 3.17 of \cite{Rodiac2016} to obtain:

\begin{lemma}\label{lem:lemme5}
Under the same assumptions as in Lemma \ref{lemme4}, there exist $0<R'<R$ and (possibly infinitely many) connected rectifiable simple curves $\Gamma_j$ such that
\begin{equation}\label{eq:decomposition}
\partial_\star B_{R'}^+ \setminus \partial B_{R'}(z_0) =\bigcup _{j=1}^{+\infty} \Gamma_j.
\end{equation}
\end{lemma}
We are now ready to prove the main theorem of this section.

\begin{proof}[Proof of Theorem \ref{th:main4} ]
Let $h\in H^1(\O)$ and $\mu \in \M(\O)$ which satisfy \eqref{1} and \eqref{eq:2complex}. Let $z_0\in \supp \mu$ such that $\nabla h(z_0) \neq 0$. We can apply Lemma \ref{lemme2}, to find  $R>0$, $\theta \in BV(B_R(z_0))$, $g=g_1+ig_2 \in W^{1,q}(B_R(z_0),\mathbb{C})$ for every $1\leq q<+\infty$ such that
\begin{equation}
g^2=(\p_zh)^2 \ \ \ \ \text{ and }   \p_zh=\theta g.
\end{equation}
Besides by using Lemma \ref{lemme3}, if we set $B_R^+=\{z\in B_R; \theta(z)=+1\}$ we have
\eqref{eq:supp} and \eqref{eq:description1}.Then Lemma \ref{lem:lemme5} provides us with $0<R'<R$ and simple connected rectifiable curves $\Gamma_j$ such that \eqref{eq:decomposition} holds. We divide the proof into three steps. \\

\textbf{Step 1:} There exists $H \in \C^{1,\a}(B_R(z_0))$ for every $0<\alpha<1$, such that $H(z_0)=0$ and $c_i\in \R$ such that $\Gamma_i=\{ z \in B_R(z_0); H(z)=c_i \}$. \\

We recall from \eqref{eq:deriveeg} that $\p_{\overline{z}}g =\theta h/4$. In particular we find $\IM (\p_{\overline{z}}g)=0$. But $\p_{\overline{z}}g=\frac12 [(\p_xg_1-\p_yg_2)+i(\p_yg_1+\p_xg_2)]$. Thus we obtain
\[ \p_yg_1=-\p_x g_2 .\]
Hence applying Poincaré's lemma we can find $H\in W^{2,q}(B_R(z_0))$ for every $1\leq q<+\infty$ such that
\[g=\frac{\p_xH-i\p_yH}{2} \text{ in } B_R(z_0),\]
and $H(z_0)=0$. Thanks to the Sobolev embeddings, we have that $H\in \C^{1,\a}(B_R(z_0))$ for every $0\leq \a <1$. Besides $h$ and $H$ are related by
\[ \nabla h= \theta \nabla H \text{ in } B_R(z_0). \]
We now show that for all $i\in \mathbb{N}$ there exists $c_i\in \R$ such that for all $z\in \Gamma_i$ we have $H(z)=c_i$. Let $i\in \mathbb{N}$, for $x,y\in \Gamma_i$ we can find $\gamma:[0,1] \rightarrow \Gamma_i$ one-to-one and Lipschitz such that $\gamma(0)=x$ and $\gamma(1)=y$, cf.\ Lemma 3 in \cite{AmbrosioCasellesMasnouMorel2001}. Since $H\in \C^1(B_R(z_0))$, we have that $H\circ \gamma \in W^{1,\infty}([0,1])$ and
\[ \frac{d}{dt}(H\circ \gamma(t))=\nabla H(\gamma(t)).\gamma'(t). \]
But by using the fundamental theorem of analysis we have
\begin{eqnarray}
H(y)-H(x)&=&\int_0^1 (H\circ\gamma)'(t)dt=\int_0^1\nabla H(\gamma(t)).\gamma'(t)dt \nonumber.
\end{eqnarray}
We have that $\gamma'(t)$ is orthogonal to $\nabla H(\gamma(t))$ for almost every $t$ in $[0,1]$ because $\nabla H$ is collinear to $\nu_{B_R^+}$ $\H^1$-almost everywhere in $\p_\star B_R^+ \setminus \p B_R$, cf.\ Lemma \ref{lemme3}. We thus have that $H(y)=H(x)$ and $\Gamma_i \subset \{z\in B_{R'}; H(z)=c_i \}$. We claim that we have the equality $\Gamma_i=\{z\in B_{R'}; H(z)=c_i \}$. Indeed, since $g$ does not vanish in $B_R(z_0)$ and $\p_xH-i\p_yH=g$ we have that $\nabla H$ does not vanish in $B_R$. Thus the sets $\{z\in B_{R'},H(z)=c_i\}$ are diffeomorphic to open intervals for $R'$ small enough. In particular these sets are regular simple connected rectifiable curves. The conclusion follows from the following lemma proved in \cite{Rodiac2016}:

\begin{lemma}\label{connectedcurve}
Let $B$ be a ball of radius $R$. Let $\gamma$ and $\tilde{\gamma}$ be two connected rectifiable simple curves. We also denote by  $\gamma, \tilde{\gamma} :[0,1] \rightarrow \R^2$ some  Lipschitz parametrizations of these curves. We suppose that $\gamma$, $\tilde{\gamma}$ are homeomorphism from $[0,1]$ onto their image. Assume that
\begin{itemize}
\item[i)] $\gamma (]0,1[) \subset B$ and $\gamma(0),\gamma(1) \in \partial B,$
\item[ii)] $\tilde{\gamma} (]0,1[) \subset B$ and $\tilde{\gamma}(0),\tilde{\gamma}(1) \in \partial B,$
\item[iii)] $\tilde{\gamma}([0,1]) \subset \gamma ([0,1]).$
\end{itemize}
Then $\gamma = \tilde{\gamma}$.
\end{lemma}

\textbf{Step 2:} There exists $\rho>0$ such that there exists a finite number of curves $\Gamma_i$ for which $\Gamma_i \cap B_\rho(z_0) \neq \emptyset$. \\
Indeed for $\rho>0$ small enough we have
\begin{eqnarray}
+\infty &>  \H^1(\supp \mu_{\lfloor B_\rho(z_0)})= & \H^1(\p_\star B_\rho^+ \setminus \p B_\rho(z_0))   \nonumber \\
&=& \H^1(\bigcup_{i=1}^{+\infty} \Gamma_i \cap B_\rho(z_0)) \nonumber \\
&\geq & \int_0^\rho \mathcal{H}^0(\bigcup_{i=1}^{+\infty}\Gamma_i \cap C_t) dt \nonumber
\end{eqnarray}
and for almost every $t\in [0,\rho]$ we have $\H^0(\bigcup_{i=1}^{+\infty} \Gamma_i \cap C_t) <+\infty$. Here we have denoted $C_t=\{z\in \R^2; |z-z_0|=t \}$, and we have used the coarea formula for the last inequality, cf.\ Theorem 2.93 in \cite{AmbrosioFuscoPallara2000}. But for $t$ small enough, every level curves $\{z \in B_t;H(z)=c_i \}$ meet the boundary of the ball $B_t$ exactly twice if it is not empty. This is due to the geometry of the level curves of a $\C^1$ function with nonvanishing gradient in a ball. We thus deduce that the number of curves $\Gamma_i$ inside $B_t$ is finite. \\

\textbf{Step 3:} Conclusion.

By applying Step 2 we can find $R$ small enough such that
\[\supp \mu_{\lfloor B_R(z_0)}=\{z\in B_R(z_0); H(z)=0\}=:\Gamma.\]
From Lemma \ref{lemme3} we have that $\mu_{\lfloor B_R(z_0)}=2(g_1\nu_1-g_2\nu_2)\H^1_{\lfloor \Gamma}.$
But since $(g_1,-g_2)=\nabla H$ we have that
\begin{equation}\label{eq:continuityofthedensity}
 \mu_{\lfloor B_R(z_0)}=2\nabla H\cdot \nu_{B_R^+} \H^1_{\lfloor \Gamma}.
\end{equation}
We recall that $\nu_{B_R^+}$ is collinear to $(-g_1,g_2)$, and since $\Gamma$ is a $\C^1$ curve, its outward pointing normal is continuous and we have, $\nu_{B_R^+}(x)=\nabla H(x)/ |\nabla H(x)|$ for every $x$ in $B_R(z_0)$ or $\nu_{B_R^+}(x)=-\nabla H(x)/|\nabla H(x)|$ for every $x$ in $B_R(z_0)$. This implies \eqref{eq:localsignofmu}.
We also have that $h$ is constant on $\supp \mu_{\lfloor B_R(z_0)}$. Indeed since $\nabla h=\theta \nabla H$ in $B_R(z_0)$ we have that there exist $c$ and $c'$ such that $h=H+c$ in $B_R^+$ and $h=-H+c'$ in $B_R^-$. But $H=0$ on $\supp \mu_{\lfloor B_R(z_0)}= \p_\star B_R^+ \setminus \p B_R(z_0)$ and $h$ is continuous in $\O$. Hence $c=c'$ and $h=c$ on $\supp \mu_{\lfloor B_R(z_0)}$. The proof is complete since Estimate \eqref{eq:BVVV} was obtained in Lemma \ref{lem:BVestimate}.
\end{proof}

\begin{corollary}\label{cor:summarysection1}
Let $h$ be in $H^1(\O)$ and $\mu$ be in $\mathcal{M}(\O)$ such that \eqref{1} and \eqref{2} hold. Then
\begin{itemize}
\item[1)] The set $A_1:=\supp \mu \cap \{| \nabla h|>0 \}$ is a countable union of $\C^1$ curves and thus is $\mathcal{H}^1$-countably rectifiable.
\item[2)] We have $\mu _{\lfloor\{|\nabla h|>0 \}}=2\sigma(x)|\nabla h|\mathcal{H}^1_{\lfloor A_1}$ with $\sigma:A_1 \rightarrow \{\pm 1 \}$ which is constant on the connected components of $A_1$.
\item[3)] The function $h$ is constant on the connected components of $A_1$.
\end{itemize}
\end{corollary}

\begin{proof}
\begin{itemize}
\item[1)] For all $z$ in $\mathbb{Q}^2 \cap A_1$ we can find $R_z>0$ such that Theorem \ref{th:main4} holds in $B_{R_z}(z)$. Since $\mathbb{Q}^2 \cap A_1$ is countable we can number its elements and obtain that $A_1 \subset \bigcup_{i=1}^{+\infty} B_{R_i}(z_i)$ for $z_i \in \mathbb{Q}^2 \cap A_1$ for all $i\in \mathbb{N}$. Now, since $\supp \mu \cap B_{R_i}(z_i)$ is a $\C^1$ curve $\gamma_i$ we obtain that $A_1=\bigcup_{i=1}^{+\infty} \gamma_i$ is a countable union of $\C^1$ curve and hence $\mathcal{H}^1$-countably rectifiable.
\item[2)] The previous point shows that if $E \subseteq \{ |\nabla h|>0 \}$ is such that $\mathcal{H}^1(E \cap A_1)=0$ then $\mu(E)=0$. This implies that $\mu_{\lfloor\{ |\nabla h|>0\}}=f\mathcal{H}^1_{\lfloor A_1}$ for some $f:A_1\rightarrow \R$ $\mathcal{H}^1$-integrable. But from \eqref{eq:localsignofmu} we see that $f=2\sigma |\nabla h|$ for some $\\sigma:A_1\rightarrow \{ \pm 1 \}$. From \eqref{eq:continuityofthedensity} we can see that $\sigma |\nabla h|$ is continuous. Thus $\sigma$ must be continuous on $A_1$ and hence constant on the connected components of $\supp \mu \cap\{ |\nabla h|>0\}$.
\item[3)] From Theorem \ref{th:main4}, $h$ is locally constant on $A_1$, since it is continuous on $\Omega$ it is constant on the connected components of $A_1$.
\end{itemize}
\end{proof}

We have seen that near regular points of $h$, the measure $\mu$ such that \eqref{1} and \eqref{2} hold is supported by a $\mathcal{C}^1$-curve. We would like to point out that we can have an intersection of several curves in the support of $\mu$ near isolated singular points of $h$. This is shown by the following example

\begin{example}
Let $n\geq 2$ and set $H_n(x):=f_n(r) \cos(n\theta)$, for $x\in  B_1(0)$ where  $(r,\theta)$ are the usual polar coordinates and $f_n$ is a solution of 
\begin{equation}
r^2f''_n(r)+rf'_n(r)-f_n(r)(r^2+n^2)=0,
\end{equation}
with $f_n(0)=0$. We can check that $f_n'(0)=0$. Thus if we set $h_n(x)=|H_n(x)|$ in $B_1(0)$ we have that $h_n$ satisfies \eqref{1} and \eqref{2} with $\mu_n:=-\Delta h_n+h_n$ which is supported on the intersection of $n$ straight lines.
\end{example}

The study of the measure $\mu$ on the set of singular points of $h$ is the content of the next section.

\section{Description of the measure $\mu$ on the set of singular points of $h$}\label{III}
 
 In this section we show that $\Delta h_{\lfloor \{|\nabla h|=0 \}}=0$ and thus $\mu_{\lfloor \{|\nabla h|=0 \}}=h\textbf{1}_{\{|\nabla h|=0 \}}$.  In order to do that, we first prove that $\nabla h$ is in $BV_{\text{loc}}(\O,\R^2)$. Then we can alternatively use an adaptation of an argument of Sandier-Serfaty in \cite{SandierSerfaty2003}, which relies on a generalized Gauss-Green formula, or use fine properties of $BV$ functions to deduce the stronger result that $D^2h_{\lfloor \{|\nabla h|=0 \}}=0$. To prove that $\nabla h$ is in $BV_{\text{loc}}(\O,\R^2)$ we use the previous section and the good estimates that we obtained on the $BV$ norm of $\nabla h$ in $\{|\nabla h|>0 \}$.

\begin{proposition}\label{prop:BV}
Let $h\in H^1(\O)$ and $\mu\in \mathcal{M}(\O)$ be such that \eqref{1} and \eqref{2} hold. Then $\nabla h$ is in $BV_{\text{loc}}(\O,\R^2)$ and 
\begin{equation}\label{eq:globalbvestimate}
\|\p_{ij}h\|_{\M(U)}\leq C\left( |\mu|(U)+ \|h\|_{L^\infty(U)}|U |\right)
\end{equation}
for $1\leq i,j\leq 2$ and for every open sets $U \subset \subset \O$.
\end{proposition}

\begin{proof}
Let $U$ be an open set such that $U\subset \subset \O$. Let $\varphi\in \C^\infty_c(U,\R^2)$ be such that $\|\varphi\|_{L^\infty}\leq 1$. We have that 
\begin{equation}
\int_U \p_xh \dive \varphi=\int_{\{ |\nabla h|>0 \}\cap U} \p_xh \dive \varphi.
\end{equation}
Now we can write \begin{equation}\nonumber
\{ |\nabla h|>0\}\cap U=\bigcup_{x\in \{|\nabla h|>0 \}\cap U} \overline{B(x,R_x)},
\end{equation}
where $R_x< \min \left(\dist(x,\{|\nabla h|=0 \}), \dist(x, \p U) \right)$ and in $B(x,R_x)$ we have \eqref{eq:BVestimate}. By using Besicovitch's covering Theorem we can find a fixed number $M>0$ and countable collections of the previous balls such that 
\begin{equation}
\{|\nabla h|>0\}\cap U=\bigcup_{i=1}^{M}\bigcup_{k_i=1}^{+\infty} \overline{B(x_{k_i},R_{k_i})}
\end{equation}
with $\overline{B(x_{k_m},R_{k_m})}\cap \overline{B(x_{k_l},R_{k_l})}=\emptyset$ if $l\neq m$.

We can assume without loss of generality that $\p_xh\dive \varphi \geq 0$ almost everywhere. Indeed if this is not true, we can replace $\varphi$ by $\tilde{\varphi}=\nabla u$ where $u$ solves $\Delta u= \text{sign} (\p_xh)^*\dive (\varphi)$. Here $(\p_xh)^*$ is the precise representative of $\p_xh$, see \cite{EvansGariepy2015} for the definition of the precise representative. We then have
\begin{eqnarray}
\int_{\{|\nabla h|>0\}\cap U} \p_xh \dive(\varphi)& =&\int_{\bigcup_{i=1}^{M}\bigcup_{k_i=1}^{+\infty} \overline{B(x_{k_i},R_{k_i})}} \p_xh \dive(\varphi) \nonumber \\
&\leq & \sum_{i=1}^M \sum_{k_i=1}^{+\infty} \int_{\overline{B(x_{k_i},R_{k_i})}} \p_xh \dive (\varphi).
\end{eqnarray}
But by using Lemma \ref{lem:BVestimate} we find that 
\begin{equation}\label{eq:estimeemesure}
\left |\int_{\overline{B(x_{k_i},R_{k_i})}} \p_xh \dive (\varphi)\right| \leq 2|\mu|(\overline{B(x_{k_i},R_{k_i}})+C\|h\|_{L^\infty(U)}|B(x_{k_i},R_{k_i})|.
\end{equation}
Indeed, for every $\a>1$ we have that 
\begin{equation}
\left |\int_{\overline{B(x_{k_i},R_{k_i})}} \p_xh \dive (\varphi)\right|  \leq \left |\int_{B(x_{k_i},\a R_{k_i})} \p_xh \dive (\phi)\right| 
\end{equation}
where $\phi$ is such that $\dive (\phi)=\dive(\varphi)$ in $\overline{B(x_{k_i},R_{k_i})}$, $\p_xh\dive(\phi)\geq 0$ almost everywhere and $\phi$ is compactly supported on $B(x_{k_i},\a R_{k_i})$. Thus 
\begin{equation}
\left |\int_{\overline{B(x_{k_i}, R_{k_i})}} \p_xh \dive (\varphi)\right| \leq |\mu|(B(x_{k_i},\a R_{k_i}))+C\|h\|_{L^\infty(U)}|B(x_{k_i},\a R_{k_i})|
\end{equation}
for any $\alpha>1$. When we let $\alpha$ go to $1$ we find \eqref{eq:estimeemesure}.
Since the balls are disjoints we can infer that 
\begin{equation}
\int_{\{|\nabla h| >0\}\cap U} \p_xh \dive(\varphi) \leq M(2|\mu|(U)+C\|h\|_{L^\infty(U)}|U|).
\end{equation}
The same is true for $\p_yh$ and thus we obtain \eqref{eq:globalbvestimate}.
\end{proof}

\begin{proposition}\label{prop:singularset}
Let $h\in H^1(\O)$ and $\mu \in \M(\O)$ be such that \eqref{1} holds. Let us assume that $|\nabla h|$ is continuous in $\O$ and $|\nabla h| \in BV(\O)$ then 
\begin{equation}
\Delta h_{\lfloor \{|\nabla h|=0 \}} =0.
\end{equation}
\end{proposition}

\begin{proof}
We follow the lines of the proof of Proposition IV.1 in \cite{SandierSerfaty2003}. We let $U_t:=\{ |\nabla h|>t \}$. Since $|\nabla h|$ is in $BV(\O)$, the coarea formula yields
\begin{equation}
\int_\O |D |\nabla h| | =\int_0^{+\infty} \mathcal{H}^1( \p U_s)ds <+\infty.
\end{equation}
See e.g.\ Theorem 1 p.185 in \cite{EvansGariepy2015}. This implies that there exists $s_n \rightarrow 0$ such that for all $n$, $U_{s_n}$ is a set of finite perimeter and $s_n\mathcal{H}^1(\p_\star U_{s_n})=s_n\mathcal{H}^1(\p_\star (U_{s_n})^c )\rightarrow 0$. Indeed by contradiction if there exists $\eta >0$ such that for all $s$ near $0$ we have $s\mathcal{H}^1(\p U_s)>\eta$, then $\mathcal{H}^1(\p U_s)>\eta/s$, and thus $\int_0^{+\infty}\mathcal{H}^1(\p U_s)$ diverges. We can take $s_n$ decreasing.
Now we note that for every Borel set $A$ we have 
\begin{equation}
\Delta h( A \cap \{ |\nabla h|=0\})=\lim_{n\rightarrow +\infty} \Delta h(A \cap \{ |\nabla h| <s_n\}).
\end{equation}
This is because the sets $\{|\nabla h|<s_n \}$ are decreasing and $|\Delta h| (\{|\nabla h |<s_n\})<+\infty$ as shown by the computations below. In particular, see Theorem 1 p.24 of \cite{EvansGariepy2015}, we have that $\nu_n:=\Delta_{\lfloor\{|\nabla h|<s_n \}}\rightharpoonup \nu:=\Delta h_{\lfloor\{|\nabla h|=0 \}}$ in the weak sense of measures. Now we take $\varphi\in \C^\infty_c(\O)$ a test function and we have, from the Gauss-Green formula on sets of finite perimeter for vector fields whose divergence is a Radon measure, cf.\ Theorem 5.2 in \cite{ChenTorresZiemer2009} and Theorem 4.4 in \cite{Silhavy2005}, 
\begin{equation}
\int_{\{|\nabla h|<s_n \}}\Delta h \varphi =-\int_{\{|\nabla h|<s_n \}} \nabla h\cdot \nabla \varphi + \int_{\p_\star \{|\nabla h|<s_n \}} \nabla h\cdot \nu d\mathcal{H}^1.
\end{equation}
In this formula $\nabla h \cdot \nu$ is a notation for the normal trace on the boundary 
$\p_\star \{|\nabla h|<s_n \}$ and must be understood in the sense of \cite{ChenTorresZiemer2009,Silhavy2005}. But these authors showed that in our particular situation, since $|\nabla h|$ is bounded we have that $\nabla h \cdot \nu$ is a $\mathcal{H}^1$-measurable function on $\p_\star \{|\nabla h|<s_n \}$ and moreover it is bounded with
\begin{equation}
\|\nabla h \cdot \nu \|_{L^\infty(\p_\star \{|\nabla h|<s_n \})} \leq  \| \nabla h \|_{L^\infty( \{|\nabla h|<s_n \})}
\end{equation}
Thus, from the choice of $s_n$  we have 
\begin{equation}
\left| \int_{\p_\star\{|\nabla h|<s_n \}} \nabla h\cdot \nu \varphi d\mathcal{H}^1 \right| \leq \|\varphi \|_{L^\infty}s_n \mathcal{H}^1(\p_\star\{|\nabla h|<s_n ) \rightarrow 0.
\end{equation}
But we also have that 
\begin{equation}
\left| \int_{\{ | \nabla h|<s_n \}} \nabla h\cdot \nabla \varphi \right|\leq C s_n \|\nabla \varphi \|_{L^\infty} \rightarrow 0.
\end{equation}
This proves that $\nu_n=\Delta_{\lfloor\{|\nabla h|<s_n \}} \rightharpoonup 0$ in the weak sense of distributions. By uniqueness of the limit we find that $\Delta h_{\lfloor \{|\nabla h|=0 \}}=0$.
\end{proof}

\begin{proposition}\label{cor:description}
Let $h\in H^1(\O)$ and $\mu\in \M(\O)$ which verify \eqref{1}, \eqref{2}. Then 
\begin{equation}\label{eq:description}
\mu=h\textbf{1}_{\{|\nabla h|=0 \}}+2\epsilon |\nabla h|\mathcal{H}^1_{\supp\mu \cap \{|\nabla h|>0 \}}
\end{equation}
for some $\epsilon:\supp\mu \cap \{|\nabla h|>0 \} \rightarrow \{ \pm 1\}$ constant on the connected components of $\supp \mu \cap \{|\nabla h|>0 \}$.
\end{proposition}

\begin{proof}
Since $h$ satisfies that $\nabla h$ is in $BV_{\text{loc}}(\O,\R^2)$ from Proposition \ref{prop:BV} we can apply Proposition \ref{prop:singularset} on every open set $U\subset \subset \O$ to obtain $\Delta h_{\lfloor \{|\nabla h|=0\cap U \}}=0$ and hence $\Delta h_{\lfloor \{|\nabla h|=0\}}=0$. We also use Corollary \ref{cor:summarysection1} of the previous section and we obtain \eqref{eq:description}.
Let us give another proof of this corollary. Since $\nabla h$ is in $BV_{\text{loc}}(\O,\R^2)$ we can apply Proposition 3.92 of \cite{AmbrosioFuscoPallara2000} to deduce that the absolute continuous part with respect to the Lebesgue measure of $D(\nabla h)$ vanishes almost everywhere on $\{|\nabla h|=0 \}$ and the Cantor part of $D(\nabla h)$ also vanishes on $\{|\nabla h|=0 \}$. But since $|\nabla h|$ is continuous in $\O$, $D(\nabla h)$ does not posses a jump part in $\{|\nabla h|=0 \}$. That is $D^2h_{\lfloor\{|\nabla h|=0 \}}=0$ and in particular $\Delta h _{\lfloor \{|\nabla h|=0 \}}=0$. Note that the proof of Proposition 3.92 in \cite{AmbrosioFuscoPallara2000} relies on the use of the coarea formula for $BV$ functions so does the proof of Proposition \ref{prop:singularset}.
\end{proof}

By using the decomposition \eqref{eq:description} we are now able to give a first proof of the fact that $h$ is constant on the connected components of $\supp \mu$. We also give two other different proofs of this result. The first one does not use that $\nabla h$ is in $BV_{\text{loc}}$ but uses only fine properties of solutions of $\Delta h\in \mathcal{M}(\O)$. The second one allows us to connect our problem with the problem studied by Caffarelli-Salazar in \cite{CaffarelliSalazar2002} and Caffarelli-Salazar-Shahgholian in \cite{CaffarelliSalazarShahgholian}.
\begin{proposition}\label{prop:const2}
Let $h$ be in $H^1(\O)$ and $\mu$ be in $\mathcal{M}(\O)$ such that \eqref{1} and \eqref{2} hold. Then $h$ is constant on the connected components of $\supp \mu$.
\end{proposition}

\begin{proof}(First proof of Proposition \ref{prop:const2})
We recall from Corollary \ref{cor:summarysection1} that $h$ is constant on the connected components of $\supp\mu \cap \{|\nabla h| >0 \}$. By using Proposition \ref{cor:description} we have that 
\begin{equation}
\Delta h=h\textbf{1}_{\{|\nabla h|>0 \}} \text { on } \O \setminus \overline{\supp \mu \cap\{ |\nabla h|>0\} }. 
\end{equation}
Thus from elliptic regularity theory we have $h$ is in $W^{2,p}_{\text{loc}}(\O \setminus \overline{\supp \mu \cap\{ |\nabla h|>0\} } )$ for every $1\leq p<+\infty$. This allows us to use a Morse-Sard type theorem, cf.\ Theorem 1.11 in \cite{dePascale2001} or Theorem 5 in \cite{Figalli2008}, to deduce that $h$ is constant on the connected components of $\{|\nabla h|=0 \}\setminus \overline {\supp \mu \cap\{ |\nabla h|>0\} }$. By continuity of $h$ we can conclude that $h$ is constant on every connected components of $\supp \mu$.
\end{proof}

 Let us now provide another proof of the previous proposition without using the decomposition \eqref{eq:description}. As mentioned before the equation $\Delta h =\nu$ is a limit case of the Calder\'on-Zygmund theory if $\nu$ is only a Radon measure. However the Calder\'on-Zygmund techniques still allow us to give some fine properties of $h$. We recall these properties.

\begin{theorem}\label{th:fine properties}(\cite{Hajlasz1996,AlbertiBianchiniCrippa2014b,PonceRodiac})
Let $h$ be in $L^1_{\text{loc}}(\O)$ and $\nu$ be in $\mathcal{M}(\O)$ such that $\Delta h=\nu$ in $\O$ in the sense of distributions, then $h$ is in $W^{1,p}_{\text{loc}}(\O)$ for every $1\leq p<2$ and 
\begin{itemize}
\item[1)] $\nabla h$ is approximately differentiable almost everywhere in $\O$.

\item[2)] We have that $(\Delta h)_a=0$ almost everywhere on $\{\nabla h=0 \}$, where $(\Delta h)_a$ is the absolutely continuous part of $\Delta h$ with respect to the Lebesgue measure. Note that since $\nabla h$ is in $L^1(\O)$ the set $\{\nabla h=0 \}$ is well-defined up to a set of zero Lebesgue measure.

\item[3)] We also have that $\nabla h$ is $L^p$-differentiable in $x$ for almost every $x$ in $\O$ and for every $1\leq p< 2$.
\end{itemize}
\end{theorem}

For the definition of the approximate differentiability and the $L^p$-differentiability we refer to \cite{AmbrosioFuscoPallara2000,EvansGariepy2015}. We note that the $L^p$-differentiability for some $p\geq 1$ implies the approximate differentiability. Item 1) was proved by Calder\'on-Zygmund, see Remark p.129 in \cite{CalderonZygmund1952} and reproved by Hajlasz in \cite{Hajlasz1996}. Item 2) is proved in \cite{PonceRodiac} by using the previous result, by showing that $(\Delta h)_a=\tr(D^2_{ap}h )$, where $D^2_{ap}$ denotes the approximate derivative of the gradient, and by using that $D^2_{ap}h=0$ almost everywhere on $\{\nabla h=0\}$, see e.g.\ \cite{EvansGariepy2015}. Item 3) is the main result of \cite{AlbertiBianchiniCrippa2014b}. It is also shown in \cite{AlbertiBianchiniCrippa2014b}, paragraph 2.4, that it implies that $h$ has the \textit{Lusin property} with functions of class $\C^2$. Namely for all $\e>0$ there exists $f_\e$ in $\C^2$ such that $|\{h\neq f_\e \}|<\e$. Thanks to that property Alberti-Bianchini-Crippa proved that if $h$ is a Lipschitz function such that $\Delta h$ is a Radon measure then $h$ has the so-called \textit{weak-Sard property}. We refer to paragraph 2.13 of \cite{AlbertiBianchiniCrippa2014a} for the definition of this notion. In our context we have the simpler fact.

\begin{proposition}\label{prop:const}
Let $h$ be a Lipschiz function in $\O$ such that $\Delta h$ is a Radon measure. Let $S_h=\{x\in \O; \nabla h(x)=0 \text{ or } h \text{ is not differentiable in x}  \}$. Then $\mathcal{L}^1(S_h)=0$ and $h$ is constant on the connected components of $S$.
\end{proposition}

\begin{proof}
This is proved in paragraph 2.15 of \cite{AlbertiBianchiniCrippa2014a}. We sketch the proof for the comfort of the reader. Let $A_n\subset \O$ and $h_n$ of class $\C^2$ in $\O$ such that $h=h_n$ on $A_n$ and $|A_n|<1/n$. Let $S_n$ be the set of critical point of $h_n$. We have that $f\cap S_n$ is contained in $S_n \cap A_n$ up to a set of $\mathcal{L}^2$-Lebesgue measure zero. Now $S$ is contained in the union of $S_n\cap A_n$ up to a negligible set. Thus $h(S)\subset h\left(\bigcup_{n=1}^{+\infty}(A_n \cap S_n) \right)\subset \bigcup_{n=1}^{+\infty} h(A_n \cap S_n)=\bigcup_{n=1}^{+\infty} h_n(A_n \cap S_n)$. From the classical Morse-Sard Theorem, since $h_n$ are $\C^2$, we have $\mathcal{L}^1(h_n(S_n))=0$. This implies that $\mathcal{L}^1(h(S))=0$. Let $C$ be a connected component of $S$ then since $h$ is continue, $h(C)$ is connected and hence is an interval. But $\mathcal{L}^1(h(C))=0$  means that $h(C)$ is a single point and $h$ is constant on $C$.
\end{proof}

\begin{proof}(Second proof of Proposition \ref{prop:const2})
From Proposition \ref{reg1} we have that $h$ is locally Lipschitz. By using Corollary \ref{cor:summarysection1}, $h$ is constant on the connected components of $\supp \mu \cap \{|\nabla h|>0 \}$. Now by applying Proposition \ref{prop:const} we also have that $h$ is constant on the connected components of $\supp \cap \{| \nabla h|=0 \}$. Since $h$ is continuous we have that $h$ is constant on every connected components of $\supp \mu$.
\end{proof}

We remark that the fine properties satisfied by a function $h$ satsifying \eqref{1} and \eqref{2} also allow us to describe the absolutely continuous part of $\mu$ with respect to the Lebesgue measure without knowing a priori that $\nabla h$ is in $BV$. We mention that for the sake of completeness. 

\begin{proposition}\label{prop:mu1}
Let $h\in H^1(\O)$ and $\mu\in \mathcal{M}(\O)$ such that \eqref{1} and \eqref{2} hold. We let $\mu=\mu_1+\mu_2$  be the Radon-Nikodym decomposition of $\mu$ with respect to the Lebesgue measure, that is $\mu_1 << \mathcal{L}^2$ and $\mu_2 \perp \mathcal{L}^2$. Then 
\begin{equation}\label{eq:mu1}
\mu_1=h\textbf{1}_{\{|\nabla h|=0 \}}\mathcal{L}^2.
\end{equation}
\end{proposition}

\begin{proof}
By Corollary \ref{cor:summarysection1} we have that $\supp \mu_1 \subset \{| \nabla h|=0 \}$. But we can apply Item 2) of Theorem \ref{th:fine properties} to obtain that ${(\Delta h)_a}_{\lfloor \{| \nabla h|=0 \}}=0$. Thus we find \eqref{eq:mu1}.
\end{proof}
The proof we have given of the previous Proposition uses the analysis of Section \ref{II}. In \cite{PonceRodiac} we also give a proof of the same fact whithout using Section \ref{II}. 

We now give a third proof of Proposition \ref{prop:const2}. In order to do that we make a link with a free boundary problem studied in \cite{CaffarelliSalazar2002,CaffarelliSalazarShahgholian}. We believe that this link can be interesting to push further the study of the free boundary $\p \supp \mu$ in our problem. We first recall their definition of solutions of $\Delta h=h$ on $\{|\nabla h|=0 \}$ in an open set $\O'$.

\begin{definition}\label{def:viscositysol1}
We say that $u$ is a \textit{subsolution} of 
\begin{equation}\label{def:viscositysol}
\Delta u=u \ \  \text {on } \ \ \{|\nabla u| \neq 0 \} 
\end{equation}
if $u$ is an upper semicontinuous function, $u<+\infty$ and $\Delta P(x) \geq u(x)$ for any  parabaloid $P$ touching $u$ from above at $x$ and such that $|\nabla P(x)|\neq 0$. A \textit{supersolution} of \eqref{def:viscositysol} is a lower semi-continuous function $u$, with $u>-\infty$ such that $\Delta P(x)\leq u(x)$ for any $x$ in $\O$ and any paraboloid $P$ touching $u$ from below and such that $|\nabla P(x)| \neq 0$. A \textit{solution} is both a super and a subsolution.
\end{definition}

 We first show that $h$ which satisfies \eqref{1},\eqref{2} is a solution of $\Delta h=h\textbf{1}_{\{|\nabla h|=0 \}}$ in the open set $\O':=\O \setminus \overline{\supp \mu\cap \{|\nabla h|>0 \}}$ in the previous sense. Then we can apply the regularity theory developed by Caffarelli-Salazar in \cite{CaffarelliSalazar2002} to obtain that $h$ is in $W^{2,p}_{\text{loc}}(\O')$ for every $1\leq p<+\infty$. This allows us to conclude with the help of a Sobolev version of  Sard's Theorem, \cite{dePascale2001,Figalli2008}, as in our first proof. The details are provided below.

\begin{proof}(Third proof of Proposition \ref{prop:const2})

 We let $\O':=\O \setminus \overline{\supp \mu \cap \{|\nabla h|>0 \}}$. In $\O'$ we can say that $h$ satsifies $\Delta h=h$ on $\{|\nabla h|\neq 0 \}$ in a viscosity sense defined by Caffarelli-Salazar in \cite{CaffarelliSalazar2002}. Indeed we first observe that since $|\nabla h|$ is continuous in $\O$ from Proposition \ref{reg1} and $h$ is smooth outside $\supp \mu$  we have that $h$ is $\C^1$ in $\O'$. Now let $P$ be a paraboloid touching $h$ from above at some $x_0$ and such that $|\nabla P(x_0)|\neq 0$. We can assume that $P-h$ has a strict local minimum at $x_0$. Since $P-h$ is $\C^1$ we infer that $|\nabla h(x_0)|\neq 0$ and thus $h$ is smooth in a small neighborhood of $x_0$. Assume by contradiction that $\Delta P(x_0)<h(x_0)$, this implies that $\Delta P(x_0)<\Delta h(x_0)$ because $h$ satisfies $\Delta h=h$ in a neighborhood of every point which are not in the support of $\mu$. But $\Delta P(x_0)<\Delta h(x_0)$ contradicts the fact that $h$ is a local minimum of $P-h$ at $x_0$. Thus $h$ is a subsolution of \eqref{def:viscositysol} and in the same way we can show that it is a supersolution, hence a solution. Now in \cite{CaffarelliSalazar2002} the authors proved that $h$ is in $W^{2,p}_{\text{loc}}(\O')$ for every $1\leq p<+\infty$. By applying the Sobolev version of Morse-Sard's Theorem \cite{dePascale2001,Figalli2008} we obtain that $h$ is constant on the connected components of $\{|\nabla h|=0 \}\cap \O'$ and hence on $\supp \mu \cap \{|\nabla h|=0 \}\cap \O'$. By continuity of $h$ and by using Corollary \ref{cor:summarysection1} we can conclude that $h$ is constant on every connected components of $\supp \mu$.
\end{proof}

\begin{proof}(Proof of Theorem \ref{th:main1})
The result is a consequence of Corollary \ref{cor:summarysection1}, Proposition \ref{cor:description}, and Proposition \ref{prop:const2}.
\end{proof}

\section{Conclusion}

We have obtained a description of the functions $h$ in $H^1(\O)$ and the Radon measures $\mu
$ which satisfy \eqref{1} and \eqref{2}. When related to the Ginzburg-Landau theory, and in 
particular the results of \cite{SandierSerfaty2003} and Chapter 13 of
 \cite{SandierSerfaty2007}, our results can be interpreted by saying that vortices of solutions of the (G.L) equations in the mean-field regime can only accumulate on ``lines" or on sets of full $\mathcal{L}^2$ Lebesgue measure if we already know that the limiting vorticity belongs to $H^{-1}(\O)$. This partially answers Open problem 14 in \cite{SandierSerfaty2007}. To give a complete answer to that problem we must also look at the case where $\mu$ is not in $H^{-1}(\O)$ and where $T_h$ defined by \eqref{eq:defTh} is divergence free in finite-part in the sense of Chapter 13 of \cite{SandierSerfaty2007}. It is then expected that $\mu$ can only charge points, ``lines" or sets of full $\mathcal{L}^2$-Lebesgue measure. We recall that solutions of the (G.L) equations whose vorticities measures converge to a measure supported on lines or on circles are known to exist, cf.\  \cite{Aydi2008}. Besides it is also known that minimizers of the (G.L) energy have vorticities measures converging to a measure which is absolutely continuous with respect to the Lebesgue measure, cf.\ Chapter 9 of \cite{SandierSerfaty2007}. We also obtained that the limiting induced magnetic field is constant on the support of the limiting vorticity and this is also expected from a physical point view. We also would like to point that our Theorem \ref{th:main4} also allows us to answers some open problems mentioned in \cite{Le2009}, where the author studied the conditions \eqref{1}, \eqref{2} assuming a priori that $\mu= f \H^1_{\lfloor \Gamma}$, where $\Gamma$ is a smooth simple closed curve and $f$ is a function in $W^{2,p}(\Gamma)$ for some $p>1$ that does not vanish. Indeed 1) we can deduce that $h$ is $\C^1$ up to the boundary of each side of $\Gamma$ just from Equation \eqref{2} without assuming that the density $f$ is in $W^{2,p}(\Gamma)$. 2) If we assume that $\mu$ is supported on just one-dimensional rectifiable curve then there is an improvement of regularity, if we know that the density is nowhere vanishing. We can deduce from Theorem  \ref{th:main4} that the curve is $\C^{1,\alpha}$ for every $0<\alpha<1$.
3) The support of $\mu$ cannot have a noninteger fractional dimension.
 
 Our second main result Theorem \ref{th:main2} is a nonexistence result for non trivial solutions of \eqref{1} \eqref{2} if we assume that $h=0$ on $\p \O$ and $\O$ is starshaped. This partially answers Open problem 18 in \cite{SandierSerfaty2007} and complement a nonexistence result in \cite{Le2009}. In the context of the (G.L) theory this results states that there cannot be critical points of the (G.L) energy with a number of vortices much larger than $h_{ex}$ if $\O$ is starshaped. 
\appendix
\section{}

In this appendix we state and prove two lemmas which are probably known to the experts but that we collect here since they are of crucial importance in this paper. The first one states that it is possible to take a logarithm of a complex-valued function which does not vanish in a ball and which is not necessarily holomorphic.

\begin{lemma}\label{lem:A1}
Let $u$ be in $W^{1,q}(B_R(z_0),\mathbb{C})$ for some $q>2$. In particular $u$ is continuous. Let us assume that $u$ does not vanish in $\overline{B_R(z_0)}$. Then there exists $v=:\log u$ in $W^{1,q}(B_R(z_0),\mathbb{C})$ such that 
\begin{equation}
u=e^v \ \ \text{ in } \ B_R(z_0).
\end{equation}
\end{lemma}

\begin{proof}
Without loss of generality we can assume that $z_0=0$. We set 
\begin{equation}
v(z):= c_0+\int_0^1 \frac{d/dt\left[u(tz) \right]}{u (tz)} dt,
\end{equation}
with $c_0 \in \mathbb{C}$ such that $e^{c_0}=u(0)$. Note that we have
\begin{equation}
v(z)=c_0+\int_0^1 \frac{x\p_xu(tz)+y\p_yu(tz)}{u(tz)}dt.
\end{equation}
We first show that $v$ is in $W^{1,q}(B_R(0),\mathbb{C})$ and that $\nabla v=\nabla u/u$ in $B_R(0)$. In order to do that we can take a sequence $(u_n)_n$ of smooth functions such that $u_n\rightarrow u $ in $W^{1,q}(B_R(0))$, and $u_n$ does not vanish in $B_R(0)$, we can also assume that there exists $a>0$  such that $|u_n|>a$, $|u|>a$ in $B_R(0)$. By using the Lebesgue dominated convergence Theorem we can see that $v_n= :\log (u_n)$ converges to $v$ in $L^q(B_R(0),\mathbb{C})$.
Since $u_n$  are smooth we can differentiate under the integral and we have that 
\begin{eqnarray}
\p_x v_n(z)&=&\int_0^1 \frac{\p_x u_n(tz))+t\p^2_{xx}u_n(tz)+ t\p^2_{xy}u_n(tz)}{u_n(tz)} \nonumber \\
& &  -\int_0^1\frac{\left[x\p_xu_n(tz)+y\p_y u_n(tz) \right]t\p_x u_n(tz)}{u_n^2(tz)} \nonumber \\
&=& \int_0^1 \frac{d}{dt}\left[ \frac{t\p_xu_n(tz)}{u_n(tz)} \right] dt \nonumber \\
&=& \frac{\p_x u_n(z)}{u_n(z)}.
\end{eqnarray}
By using the dominated convergence again, and the fact that $|u_n|>a$ we can show that $\p_x v_n$ converges in $L^q(B_R(0),\mathbb{C})$ towards $\p_x u/u$. The computations being similar for $\p_y v_n$ we obtain that $v$ is in $W^{1,q}(B_R(0),\mathbb{C})$ and $\nabla v=\nabla u/ u$ in that ball. Now we can use the chain rule in Sobolev spaces since exp is smooth and we obtain that $\nabla (e^v)= \nabla v e^v=\frac{\nabla u }{u}u=\nabla u$. But since $u(0)=e^{c_0}=e^{v(0)}$ we find that $u(z)=e^z$ in $B_R(0)$.
\end{proof}

The second lemma shows that the Leibniz rule is valid for a product of a BV function and a Sobolev function under appropriate hypotheses, this is proved in Example 3.97 of \cite{AmbrosioFuscoPallara2000}.

\begin{lemma}\label{lem:A2}
Let $u$ be in $BV\cap L^\infty(\O)$ and $v$ be in $\C^0(\overline{\O})\cap W^{1,p}(\O)$ for some $1\leq p\leq +\infty$. Then $(uv)$ is in $BV(\O)$ and 
\begin{equation}
D(uv)=vDu+u\nabla v \text{  in  } \O.
\end{equation}
\end{lemma}

\begin{proof}
Let $(u_n)_n$, $(v_n)_n$ be sequences of smooth functions with compact support in $\O$ such that 
\begin{itemize}
\item[i)] $u_n \rightarrow u $ in $L^1(\O)$, $\nabla u_n \rightharpoonup Du$ in $\M(\O)$ and $\int_\O |\nabla u_n|\rightarrow|Du|(\O)$.
\item[ii)] $v_n \rightarrow v$ in $\C^0(\overline{\O})$  and $\nabla v_n\rightarrow \nabla v$ in $L^p(\O)$,
\item[iii)] we can also assume that $\|u_n\|_{L^\infty(\O)}, \| v_n\|_{L^\infty(\O)}\leq M$, $\|u\|_{L^\infty(\O)}, \|v\|_{L^\infty(\O)} \leq M$ for some $M>0$.
\end{itemize}
By applying the dominated convergence Theorem we can see that $u_n v_n \rightarrow u v$ in $L^1(\O)$. We also have that $\nabla (u_n v_n)=v_n \nabla u_n+ u_n\nabla v_n$ in $\O$. 
Let $\varphi\in \C^\infty_c(\O)$,
\begin{equation}\nonumber
\int_\O v_n \nabla u_n \varphi -\int_\O v \varphi Du =\int_\O (v_n-v)\nabla u_n \varphi +\int_\O v \varphi (\nabla u_n -Du). 
\end{equation}
Now we can see that 
\begin{eqnarray}
\left| \int_\O (v_n-v)\nabla u_n  \varphi \right| & \leq & \| v_n-v\|_{L^\infty(\O)}\|\varphi\|_{L^\infty(\O)}\int_\O |\nabla u_n|. \nonumber \\
\end{eqnarray}
But since $\int_\O |\nabla u_n| \rightarrow |Du|(\O)$ we have that $\left| \int_\O (v_n-v)\nabla u_n  \varphi \right|\rightarrow 0$ as $n\rightarrow +\infty$. We also have that 
\begin{equation}
\int_\O g\varphi (\nabla u_n-Du)\rightarrow 0,
\end{equation}
because $\nabla u_n\rightharpoonup Du$ in $\M(\O)$ and $g\varphi \in \C^0_c(\O)$. This proves that $\int_\O v_n\nabla u_n \varphi \rightarrow \int_\O v \varphi Du$ as $n\rightarrow +\infty$.
Now we also have that 
\begin{equation}
\int_\O |u_n\nabla v_n -u\nabla v|^p \leq C\int_\O |(u_n-u)\nabla v_n|^p+C\int_\O |\nabla v_n-\nabla v|^p |u|^p,
\end{equation} 
By using the dominated convergence theorem we find that $\int_\O |u_n\nabla v_n -u\nabla v|^p\rightarrow 0$.
Thus we can conclude that 
\begin{equation}
\nabla (u_n v_n) \rightharpoonup vDu+ u\nabla v \ \ \text{ in } \mathcal{M}(\O). 
\end{equation}
In particular the derivative in the sense of distributions of $uv$ is $vDu+ u\nabla v$ and since it is a Radon measure we have that $uv$ is in $BV(\O)$.

\end{proof}

\textbf{Acknowledgements:} I would like to express my special thanks to Etienne Sandier for useful conversations on the topic of this article. I also thank Duvan Henao, Xavier Lamy and Nam Q. Le for their comments on a first version of this paper.
\bibliographystyle{abbrv}
\bibliography{biblio}

\begin{thebibliography}{10}

\bibitem{AlbertiBianchiniCrippa2014b}
G.~Alberti, S.~Bianchini, and G.~Crippa.
\newblock On the {$L^p$}-differentiability of certain classes of functions.
\newblock {\em Rev. Mat. Iberoam.}, 30(1):349--367, 2014.

\bibitem{AlbertiBianchiniCrippa2014a}
G.~Alberti, S.~Bianchini, and G.~Crippa.
\newblock A uniqueness result for the continuity equation in two dimensions.
\newblock {\em J. Eur. Math. Soc. (JEMS)}, 16(2):201--234, 2014.

\bibitem{AlikakosFaliagas2012}
N.~D. Alikakos and A.~C. Faliagas.
\newblock The stress-energy tensor and {P}ohozaev's identity for systems.
\newblock {\em Acta Math. Sci. Ser. B Engl. Ed.}, 32(1):433--439, 2012.

\bibitem{AmbrosioCasellesMasnouMorel2001}
L.~Ambrosio, V.~Caselles, S.~Masnou, and J.-M. Morel.
\newblock Connected components of sets of finite perimeter and applications to
  image processing.
\newblock {\em J. Eur. Math. Soc. (JEMS)}, 3(1):39--92, 2001.

\bibitem{AmbrosioFuscoPallara2000}
L.~Ambrosio, N.~Fusco, and D.~Pallara.
\newblock {\em Functions of bounded variation and free discontinuity problems}.
\newblock Oxford Mathematical Monographs. The Clarendon Press, Oxford
  University Press, New York, 2000.

\bibitem{Aydi2008}
H.~Aydi.
\newblock Lines of vortices for solutions of the {G}inzburg-{L}andau equations.
\newblock {\em J. Math. Pures Appl. (9)}, 89(1):49--69, 2008.

\bibitem{BethuelBrezisHelein1994}
F.~Bethuel, H.~Brezis, and F.~H\'elein.
\newblock {\em Ginzburg-{L}andau vortices}, volume~13 of {\em Progress in
  Nonlinear Differential Equations and their Applications}.
\newblock Birkh\"auser Boston, Inc., Boston, MA, 1994.

\bibitem{CaffarelliSalazar2002}
L.~Caffarelli and J.~Salazar.
\newblock Solutions of fully nonlinear elliptic equations with patches of zero
  gradient: existence, regularity and convexity of level curves.
\newblock {\em Trans. Amer. Math. Soc.}, 354(8):3095--3115, 2002.

\bibitem{CaffarelliSalazarShahgholian}
L.~Caffarelli, J.~Salazar, and H.~Shahgholian.
\newblock Free-boundary regularity for a problem arising in superconductivity.
\newblock {\em Arch. Ration. Mech. Anal.}, 171(1):115--128, 2004.

\bibitem{Caffarelli1998}
L.~A. Caffarelli.
\newblock The obstacle problem revisited.
\newblock {\em J. Fourier Anal. Appl.}, 4(4-5):383--402, 1998.

\bibitem{CalderonZygmund1952}
A.~P. Calder\'on and A.~Zygmund.
\newblock On the existence of certain singular integrals.
\newblock {\em Acta Math.}, 88:85--139, 1952.

\bibitem{ChapmanRubinsteinSchatzman1996}
S.~J. Chapman, J.~Rubinstein, and M.~Schatzman.
\newblock A mean-field model of superconducting vortices.
\newblock {\em European J. Appl. Math.}, 7(2):97--111, 1996.

\bibitem{ChenTorresZiemer2009}
G.-Q. Chen, M.~Torres, and W.~P. Ziemer.
\newblock Gauss-{G}reen theorem for weakly differentiable vector fields, sets
  of finite perimeter, and balance laws.
\newblock {\em Comm. Pure Appl. Math.}, 62(2):242--304, 2009.

\bibitem{ContrerasSerfaty2012}
A.~Contreras and S.~Serfaty.
\newblock Large vorticity stable solutions to the {G}inzburg-{L}andau
  equations.
\newblock {\em Indiana Univ. Math. J.}, 61(5):1737--1763, 2012.

\bibitem{dePascale2001}
L.~de~Pascale.
\newblock The {M}orse-{S}ard theorem in {S}obolev spaces.
\newblock {\em Indiana Univ. Math. J.}, 50(3):1371--1386, 2001.

\bibitem{Delort1991}
J.-M. Delort.
\newblock Existence de nappes de tourbillon en dimension deux.
\newblock {\em J. Amer. Math. Soc.}, 4(3):553--586, 1991.

\bibitem{DiPernaMajda1988}
R.~J. DiPerna and A.~Majda.
\newblock Reduced {H}ausdorff dimension and concentration-cancellation for
  two-dimensional incompressible flow.
\newblock {\em J. Amer. Math. Soc.}, 1(1):59--95, 1988.

\bibitem{EvansGariepy2015}
L.~C. Evans and R.~F. Gariepy.
\newblock {\em Measure theory and fine properties of functions}.
\newblock Textbooks in Mathematics. CRC Press, Boca Raton, FL, revised edition,
  2015.

\bibitem{Faliagas2016}
A.~C. Faliagas.
\newblock On the equivalence of {E}uler-{L}agrange and {N}oether equations.
\newblock {\em Math. Phys. Anal. Geom.}, 19(1):Art. 1, 12, 2016.

\bibitem{Figalli2008}
A.~Figalli.
\newblock A simple proof of the {M}orse-{S}ard theorem in {S}obolev spaces.
\newblock {\em Proc. Amer. Math. Soc.}, 136(10):3675--3681, 2008.

\bibitem{Frehse1972}
J.~Frehse.
\newblock On the regularity of the solution of a second order variational
  inequality.
\newblock {\em Boll. Un. Mat. Ital. (4)}, 6:312--315, 1972.

\bibitem{GilbargTrudinger}
D.~Gilbarg and N.~S. Trudinger.
\newblock {\em Elliptic partial differential equations of second order}.
\newblock Classics in Mathematics. Springer-Verlag, Berlin, 2001.
\newblock Reprint of the 1998 edition.

\bibitem{Hajlasz1996}
P.~Haj\l~asz.
\newblock On approximate differentiability of functions with bounded
  deformation.
\newblock {\em Manuscripta Math.}, 91(1):61--72, 1996.

\bibitem{IwaniecKovalevOnninen2013}
T.~Iwaniec, L.~V. Kovalev, and J.~Onninen.
\newblock Lipschitz regularity for inner-variational equations.
\newblock {\em Duke Math. J.}, 162(4):643--672, 2013.

\bibitem{Le2009}
N.~Q. Le.
\newblock Regularity and nonexistence results for some free-interface problems
  related to {G}inzburg-{L}andau vortices.
\newblock {\em Interfaces Free Bound.}, 11(1):139--152, 2009.

\bibitem{Pohozaev1965}
S.~I. Pohozaev.
\newblock On the eigenfunctions of the equation {$\Delta u+\lambda f(u)=0$}.
\newblock {\em Dokl. Akad. Nauk SSSR}, 165:36--39, 1965.

\bibitem{Ponce2016}
A.~C. Ponce.
\newblock {\em Elliptic {PDE}s, measures and capacities}, volume~23 of {\em EMS
  Tracts in Mathematics}.
\newblock European Mathematical Society (EMS), Z\"urich, 2016.
\newblock From the Poisson equations to nonlinear Thomas-Fermi problems.

\bibitem{PonceRodiac}
A.~C. Ponce and R.~Rodiac.
\newblock Laplacian vanishes almost everywhere on level sets.
\newblock in preparation.

\bibitem{Rivierecours}
T.~Rivi\`ere.
\newblock Conformally invariant variational problems.
\newblock Lecture notes available at
  https://people.math.ethz.ch/~riviere/papers/conformal-course.pdf.

\bibitem{Riviere1995}
T.~Rivi\`ere.
\newblock Everywhere discontinuous harmonic maps into spheres.
\newblock {\em Acta Math.}, 175(2):197--226, 1995.

\bibitem{Rodiac2016}
R.~Rodiac.
\newblock Regularity properties of stationary harmonic functions whose
  {L}aplacian is a {R}adon measure.
\newblock {\em SIAM J. Math. Anal.}, 48(4):2495--2531, 2016.

\bibitem{SandierSerfaty2003}
E.~Sandier and S.~Serfaty.
\newblock Limiting vorticities for the {G}inzburg-{L}andau equations.
\newblock {\em Duke Math. J.}, 117(3):403--446, 2003.

\bibitem{SandierSerfaty2007}
E.~Sandier and S.~Serfaty.
\newblock {\em Vortices in the magnetic {G}inzburg-{L}andau model}, volume~70
  of {\em Progress in Nonlinear Differential Equations and their Applications}.
\newblock Birkh\"auser Boston, Inc., Boston, MA, 2007.

\bibitem{Silhavy2005}
M.~Silhav\'y.
\newblock Divergence measure fields and {C}auchy's stress theorem.
\newblock {\em Rend. Sem. Mat. Univ. Padova}, 113:15--45, 2005.

\bibitem{Watson}
G.~N. Watson.
\newblock {\em A treatise on the theory of {B}essel functions}.
\newblock Cambridge Mathematical Library. Cambridge University Press,
  Cambridge, 1995.
\newblock Reprint of the second (1944) edition.

\end{thebibliography}

\end{document}